\RequirePackage{fix-cm}
\documentclass[smallextended]{svjour3}       
\smartqed  
\usepackage{graphicx}
\usepackage{fullpage}
\usepackage{color}
\usepackage{amsmath}
\allowdisplaybreaks
\usepackage{algorithm}
\usepackage{algorithmic}
\usepackage{graphicx,epstopdf}
\usepackage{appendix}
\usepackage{booktabs}
\usepackage{tikz}
\usepackage{epstopdf} 
\newtheorem{exmp}{Example}
\def\sD{\mathcal{D}}

\def\sN{\mathcal{N}}

\def\bE{\mathbb{E}}
\def\bR{\mathbb{R}}

\def\sI{\mathcal{I}}
\def\sE{\mathcal{E}}

\def\sF{\mathcal{F}}

\def\sI{\mathcal{I}}

\def\snk{\mathcal{S}_k}

\def\dnk{d^k}
\def\lnk{\lambda_k}
\def\mnk{\mu_k}
\def\pnk{P}
\def\tnk{t_k}

\def\argmin{\mbox{argmin}}
\def\rg{\mbox{range}}

\def\sX{\mathcal{X}}

\usepackage{amsmath, amssymb, bbm, xspace}

\DeclareMathOperator*{\st}{subject\;to}

\def\rank{\mathop{\hbox{\rm rank}}}

\def\spose#1{\hbox to 0pt{#1\hss}}

\def\text #1{\hbox{\quad#1\quad}}

\def\nthinsp{\mskip -2   mu}

\def\superstar{^{\raise 0.5pt\hbox{$\nthinsp *$}}}
\def\SUPERSTAR{^{\raise 0.5pt\hbox{$*$}}}

\def\lamstarT {\lambda^{\raise 0.5pt\hbox{$\nthinsp *$}T}}

\def\hbar{\skew{4.2}\bar h}

		\def\bkE{{\rm I\kern-.17em E}}
		\def\bkE{\mathbb{E}}
		\def\bk1{{\rm 1\kern-.17em l}}
		\def\bkD{{\rm I\kern-.17em D}}
		\def\bkR{{\rm I\kern-.17em R}}
		\def\bkP{{\rm I\kern-.17em P}}
		\def\bkY{{\bf \kern-.17em Y}}
		\def\bkZ{{\bf \kern-.17em Z}}

		\def\beq{\begin{eqnarray}}
		\def\bc{\begin{center}}
		\def\be{\begin{enumerate}}
		\def\bi{\begin{itemize}}
		\def\bs{\begin{small}}
		\def\bS{\begin{slide}}
		\def\ec{\end{center}}
		\def\ee{\end{enumerate}}
		\def\ei{\end{itemize}}
		\def\es{\end{small}}
		\def\eS{\end{slide}}
		\def\eeq{\end{eqnarray}}
		\def\qed{\quad \vrule height7.5pt width4.17pt depth0pt}

	\def\cp2problem#1#2#3#4{\fbox
		 {\begin{tabular*}{0.9\textwidth}
			{@{}l@{\extracolsep{\fill}}l@{\extracolsep{6pt}}l@{\extracolsep{\fill}}c@{}}
				#1 & & $#4 $ 
			\end{tabular*}}}

\newcommand{\pmat}[1]{\begin{pmatrix} #1 \end{pmatrix}}
		
		\renewcommand{\emph}[1]{\textbf{#1}}

		\def\bk1{{\rm 1\kern-.17em l}}
		\def\bkD{{\rm I\kern-.17em D}}
		\def\bkR{{\rm I\kern-.17em R}}
		\def\bkP{{\rm I\kern-.17em P}}
		
		\def\bkZ{{\bf{Z}}}

\newcommand {\beeq}[1]{\begin{equation}\label{#1}}
\newcommand {\eeeq}{\end{equation}}
\newcommand {\bea}{\begin{eqnarray}}
\newcommand {\eea}{\end{eqnarray}}

\def\texitem#1{\par\smallskip\noindent\hangindent 25pt
               \hbox to 25pt {\hss #1 ~}\ignorespaces}

\def\st{\mbox{subject to}}

\def\EA{\mathcal{E}_A}

\def\sinkhorn{\textbf{Sinkhorn}}
\def\arbcd{{\textbf{ARBCD}}}
\def\rbcd{\textbf{RBCD$_0$}}
\def\rbcddb{\textbf{RBCD-DB}}
\def\rbcdsdb{\textbf{RBCD-SDB}}
\DeclareMathOperator{\nul}{null}
\DeclareMathOperator{\supp}{supp}
\DeclareMathOperator{\ve}{vec}
\DeclareMathOperator{\mo}{mod}
\def\rev#1{\textcolor{black}{#1}}
\def\he{\hat{\epsilon}}
\def\ipm{{\textbf{IPM}}}
\begin{document}

\title{Randomized methods for computing optimal transport without regularization and their convergence analysis
}


\author{Yue Xie         \and
        Zhongjian Wang \and 
        Zhiwen Zhang
}


\institute{Yue Xie\at
              Department of Mathematics,
              The University of Hong Kong, Pokfulam Road, Hong Kong SAR, China \\
              \email{yxie21@hku.hk}           
           \and
           Zhongjian Wang \at
              School of Physical and Mathematical Sciences, Nanyang Technological University, 21 Nanyang Link, Singapore 637371 \\
              \email{zhongjian.wang@ntu.edu.sg}
              \and
        Zhiwen Zhang \at 
        Department of Mathematics, The University of Hong Kong, Pokfulam Road, Hong Kong SAR, China \\
        \email{zhangzw@hku.hk}
}

\date{Received: date / Accepted: date}

\maketitle

\begin{abstract}
The optimal transport (OT) problem can be reduced to a linear programming (LP) problem through discretization. In this paper, we introduced the random block coordinate descent (RBCD) methods to directly solve this LP problem. Our approach involves restricting the potentially large-scale optimization problem to small LP subproblems constructed via randomly chosen working sets. By using a random Gauss-Southwell-$q$ rule to select these working sets, we equip the vanilla version of (\rbcd) with almost sure convergence and a linear convergence rate to solve general standard LP problems. To further improve the efficiency of the (\rbcd) method, we explore the special structure of constraints in the OT problems and leverage the theory of linear systems to propose several approaches for refining the random working set selection and accelerating the vanilla method. Inexact versions of the RBCD methods are also discussed. Our preliminary numerical experiments demonstrate that the accelerated random block coordinate descent (\arbcd) method compares well with other solvers including Sinkhorn's algorithm when seeking solutions with relatively high accuracy, and offers the advantage of saving memory. 
\end{abstract}

\section{Introduction} \label{sec:introduction.}
\paragraph{Background and motivation} The optimal transport problem was first introduced by Monge in 1781, which aims to find the most cost-efficient way to transport mass from a set of sources to a set of sinks. Later, the theory was modernized and revolutionized by Kantorovich in 1942, who found a key link between optimal transport and linear programming. In recent years, optimal transport has become a popular and powerful tool in data science, especially in image processing, machine learning, and deep learning areas, where it provides a very natural way to compare and interpolate probability distributions. For instance, in generative models \cite{arjovsky2017wasserstein,lei2019geometric,wang2022deepparticle}, a natural penalty function is the Wasserstein distance (a concept closely related to OT) between the data and the generated distribution. In image processing, the optimal transport plan, which minimizes the transportation cost, provides solutions to image registration \cite{haker2004optimal} and seamless copy \cite{perrot2016mapping}.  Apart from data science, in the past three decades, there has been an explosion of research interest in optimal transport because of the deep connections between the optimal transport problems with quadratic cost functions and a diverse class of partial differential equations (PDEs) arising in statistical mechanics and fluid mechanics; see e.g. \cite{brenier1991polar,benamou2000computational,otto2001geometry,jordan1998variational,villani2021topics} for just a few of the most prominent results and references therein.

Inspired by this research progress, we have developed efficient numerical methods to solve multi-scale PDE problems using the optimal transport approach. Specifically, in our recent paper, we proposed a deep particle method for learning and computing invariant measures of parameterized stochastic dynamical systems \cite{wang2022deepparticle}. To achieve this goal, we designed a deep neural network (DNN) to map a uniform distribution (source) to an invariant measure (target),  where the P\'eclet number is an input parameter for the DNN. The network is trained by minimizing the 2-Wasserstein distance ($W_2$) between the measure of network output  $\mu$ and target measure $\nu$. We consider a discrete version of $W_2$ for finitely many samples of $\mu$ and $\nu$, which involves a linear program (LP) optimized over doubly stochastic matrices \cite{sinkhorn1964relationship}.

Solving the LP directly using the interior point method \cite{wright1997primal} is too costly. Motivated by the domain decomposition method \cite{toselli2004domain} in scientific computing, which solves partial differential equations using subroutines that solve problems on subdomains and has the advantage of saving memory (i.e., using the same computational resource, it can compute a larger problem), we devised a mini-batch interior point method. This approach involves sampling smaller sub-matrices while preserving row and column sums. It has proven to be highly efficient and integrates seamlessly with the stochastic gradient descent (SGD) method for overall network training. However, we did not obtain convergence analysis for this mini-batch interior point method in our previous work \cite{wang2022deepparticle}.

The objectives of this paper are twofold. First, we aim to provide rigorous convergence analysis for the mini-batch interior point method presented in \cite{wang2022deepparticle}, with minimal modifications. Second, we seek to enhance the mini-batch selection strategy, thereby achieving improved and more robust performance in computing optimal transport problems. We recognize that the mini-batch interior point method aligns with the random block coordinate descent (RBCD) method in optimization terminology. Specifically, it applies the block coordinate descent (BCD) method to the LP problem directly, selects the working set randomly, and solves subproblems using the primal-dual interior point method \cite{wright1997primal} or any other efficient linear programming solver. Encouraged by the demonstrated efficiency of this approach, we will develop theoretical results for solving LP with RBCD methods and explore various strategies for selecting working sets.

\paragraph{Theorectical contributions} In this work, we first introduce an expected Gauss-Southwell-$q$ rule to guide the selection of the working set. It enables almost sure convergence and a linear convergence rate in expectation when solving a general standard LP. Based on this rule, we develop a vanilla RBCD method - {\rbcd}, which selects the working set with complete randomness. Then, we investigate the special linear system present in the LP formulation of OT. We characterize all the elementary vectors of the null space and provide a strategy for finding the conformal realization of any given vector in the null space at a low computational cost. Based on these findings, we propose various approaches to refine the working set selection and improve the performance of {\rbcd}. A better estimation of the constant in the linear convergence rate is shown. Moreover, we incorporate an acceleration technique inspired by the momentum concept to improve the algorithm's efficiency. Inexact versions of the RBCD methods are also discussed.

\paragraph{Numerical experiments}
We perform numerical experiments to evaluate the performance of the proposed methods. Synthetic data sets of various shapes/dimensions and invariant measures generated from IPM methods are utilized to create distributions. Our experiments first compare different RBCD methods proposed in this paper, demonstrating the benefits of refining working set selection and verifying the effectiveness of the acceleration technique. We also illustrate the gap between theory and practice regarding convergence rate, sparse solutions generated by the proposed RBCD methods, and discuss the choice of subproblem size. In the second set of experiments, we compare the best-performance method, {\arbcd}, with Sinkhorn's algorithm. Preliminary numerical results show that {\arbcd} is comparable to Sinkhorn's algorithm in computation time when seeking solutions with relatively high accuracy. \rev{{\arbcd} is also comparable to a recently proposed interior point inspired algorithm in memory-saving settings.} We also test {\arbcd} on a large-scale OT problem, where Gurobi runs out of memory. This further justifies the memory-saving advantage of {\arbcd}.

\paragraph{Previous research on (R)BCD} BCD and RBCD are well-studied for essentially unconstrained smooth optimization (sometimes allow separable constraints or nonsmooth separable objective functions): \cite{beck2013convergence,gurbuzbalaban2017cyclic,sun2021worst} investigate BCD with cyclic coordinate search; \cite{nesterov2012efficiency,lu2015complexity,richtarik2014iteration} study RBCD to address problems with possibly nonsmooth separable objective functions; other related works include theoretical speedup of RBCD (\cite{richtarik2016parallel,necoara2016parallel}), second-order sketching (\cite{qu2016sdna,berahas2020investigation}). However, much less is known for their convergence properties when applied to problems with nonseparable nonsmooth functions as summands or coupled constraints. To our best knowledge, no one has ever considered using the RBCD to solve general LP before and the related theoretical guarantees are absent. In \cite{necoara2017random}, the authors studied the RBCD method to tackle problems with a convex smooth objective and coupled linear equality constraints $x_1 + x_2 + \hdots + x_N = 0$; a similar algorithm named random sketch descent method \cite{necoara2021randomized} is investigated to solve problems with a general smooth objective and general coupled linear equality constraints $Ax = b$. However, after adding the simple bound constraints $x \ge 0$, the analysis in \cite{necoara2017random,necoara2021randomized} may not work anymore, nor can it be easily generalized. Beck \cite{beck20142} studied a greedy coordinate descent method but focused on a single linear equality constraint and bound constraints. In Paul Tseng and his collaborators' work \cite{tseng2009block,tseng2009coordinate,tseng2010coordinate}, a block coordinate gradient descent method is proposed to solve linearly constrained optimization problems including general LP. In these works, a Gauss-Southwell-$q$ rule is proposed to guide the selection of the working set in each iteration. Therefore, the working set selected in a deterministic fashion can only be decided after solving a quadratic program with a similar problem size as the original one. In contrast, our proposed mini-batch interior point/RBCD method approach selects the working set
through a combination of randomness and low computational cost. Another research direction that addresses separable functions, linearly coupled constraints, and additional separable constraints involves using the alternating direction method of multipliers (ADMM) \cite{chen2016direct,he20121,xie2019si,xie2021tractable}. This method updates blocks of primal variables in a Gauss-Seidal fashion and incorporates multiplier updates as well.


\paragraph{Existing algorithms for OT} \rev{Encouraged by the success in applying Sinkhorn's algorithm to the dual of entropy regularized OT \cite{cuturi2013sinkhorn}, researchers have conducted extensive studies in this area, including other types of regularization \cite{blondel2018smooth}\cite{gasnikov2016efficient}, acceleration \cite{guminov2021combination}\cite{lin2022efficiency} and numerical stability \cite{mandad2017variance}. In \cite{xie2020fast}, a proximal point algorithm (PPA) is considered to solve the discrete LP. The entropy inducing distance function is introduced to contruct the proximal term and each subproblem has the same formulation as the entropy regularized OT. This approach is found to stabilize the numerical computation while maintaining the efficiency of the Sinkhorn's algorithm. In \cite{schmitzer2019stabilized}, techniques such as log-domain stabilization, and epsilon scaling are discussed to further improve the performance of the Sinkhorn's algorithm. The approach of iterative Bregman projection is discussed in \cite{benamou2015iterative}. It is discovered that the Sinkhorn's iteration is equivalent to a special case of this method. In \cite{pmlr-v84-genevay18a}, a notion of Sinkhorn distance is proposed, which allows computable differentiation when serving as the loss function when training generative models. Sinkhorn's algorithm can also be generalized to solve unbalanced OT in \cite{chizat2018scaling}, and applied to solve OT over geometric domains in \cite{solomon2015convolutional}. In \cite{huang2021riemannian}, a Riemannian block coordinate descent method is applied to solve projection robust Wasserstein distance. The problem is to calculate Wasserstein distance after projecting the distributions onto lower-dimensional subspaces. The proposed approach employs entropy regularization, deterministic block coordinate descent, and techniques in Riemannian optimization. A domain decomposition method is considered in \cite{bonafini2021domain}. Unlike our method, the decomposition is deterministic, focusing on one pair of distributions, and the authors try to tackle entropy-regularized OT.}

\rev{Other works that significantly deviate from the entropy regularization framework include \cite{li2018computations}, which computes the Schr\"{o}dinger bridge problem (equivalent to OT with Fisher information regularization), and multiscale strategies such as \cite{gerber2017multiscale},  \cite{liu2022multiscale} and \cite{schmitzer2016sparse}. In \cite{liu2022multiscale}, the problem size is reduced using a multiscale strategy, and a semi-smooth Newton method is applied to solve smaller-scale subproblems. The RBCD method employed in this study is a regularization-free method. As a result, it avoids dealing with inaccurate solutions and numerical stability issues introduced by the regularization term. Furthermore, each subproblem in RBCD is a small-size LP, allowing for flexible resolution choices. Interior-point methods and simplex methods are also revisited and enhanced \cite{zanetti2023interior,wijesinghe2023matrix,natale2021computation,gottschlich2014shortlist}. In particular, the authors of \cite{zanetti2023interior} propose an interior point inspired algorithm with reduced subproblems to address large-scale OT. The authors of \cite{facca2021fast} consider an equivalent formulation of minimizing an energy functional to solve OT over graphs. Backward Euler (equivalent to the proximal point method) is used in the outer loop of the algorithm and Newton-Rapson is used in the inner loop. Other approaches include numerical solution of 
the Monge--Amp{\`e}re equation
\cite{benamou2014numerical,benamou2016monotone}, stochastic gradient descent to resolve OT in discrete, semi-discrete and continuous formulations \cite{genevay2016stochastic}, and an alternating direction method of multipliers (ADMM) approach to solve the more general Wasserstein barycenter \cite{yang2021fast}.}

\paragraph{Organization} The rest of the paper is organized as follows. In Section~\ref{sec:WdistanceOT}, we review the basic idea of optimal transport and Wasserstein distance. In Section~\ref{sec: converg.}, we introduce the expected Gauss-Southwell-$q$ rule and a vanilla RBCD (\rbcd) method for solving general LP problems. An inexact version of {\rbcd} is also discussed. In Section~\ref{sec:RBCD-OT}, we investigate the properties of the linear system in OT and propose several approaches to refine and accelerate the {\rbcd} method. In Section~\ref{sec:num}, preliminary numerical results are presented to demonstrate the performance of our proposed methods. Finally, concluding remarks are made in Section~\ref{sec:conclusion}.

\noindent{\it Notation. } For any matrix $X$, let $X(i,j)$ denote its element in the $i$th column and $j$th row, and let $X(:,j)$ represent its $j$th row vector. For a vector $v$, we usually use superscripts to denote its copies (e.g., $v^k$ in $k$th iteration of an algorithm) and use subscripts to denote its components (e.g., $v_i$); for a scalar, we usually use subscripts to denote its copies. Occasional inconsistent cases will be declared in context. $\mo(k,n)$ means $k$ modulo $n$. For any vector $v$, we define $\supp(v) \triangleq \{ i \in \{1,\hdots,n \} \mid v_i \neq 0 \}$. Given a matrix $X \in \bR^{n \times n}$, we define its vectorization as follows:
\begin{align*} 
    \ve(X) \triangleq (X(:,1)^T,X(:,2)^T,...,X(:,n)^T)^T.
\end{align*}
For any positive integer $k\ge 2$, we denote $[1,k] \triangleq \{1,...,k\}$. ${\bf 1}_{n \times n}$ represents the $n \times n$ matrix of all ones.

\section{Optimal transport problems and Wasserstein distance} \label{sec:WdistanceOT}

	
The Kantorovich formulation of optimal transport can be described as follows,
\begin{align}\label{def:OT}
\inf _{\gamma \in \Gamma(\mu, \nu)} \int_{X \times Y} C(x, y) \mathrm{~d} \gamma(x, y)
	\end{align}
where $\Gamma(\mu,\nu)$ is the set of all measures on $X \times Y$ whose marginal distribution on $X$ is $\mu$ and marginal distribution on $Y$ is $\nu$, $C(x,y)$ is the transportation cost. In this article, we refer to the Kantorovich formulation when we mention optimal transport.

Wasserstein distances are metrics on probability distributions inspired by the problem of optimal mass transport. They measure the minimal effort required to reconfigure the probability mass of one distribution in order to recover the other distribution. They are ubiquitous in mathematics, especially in fluid mechanics, PDEs, optimal transport, and probability theory \cite{villani2021topics}. One can define the $p$-Wasserstein distance between probability measures $\mu$ and $\nu$ on a metric space $Y$ with distance function $dist$ by 

	\begin{align}\label{def:p-Wdistance}
	W_{p}(\mu, \nu):=\left(\inf _{\gamma \in \Gamma(\mu, \nu)} \int_{Y \times Y} dist (\tilde y, y)^{p} \mathrm{~d} \gamma(\tilde y, y)\right)^{1 / p}
	\end{align}
	where $\Gamma(\mu, \nu)$ is the set of probability measures $\gamma$ on $Y\times Y$ satisfying $\gamma(A\times Y)=\mu(A)$ and $\gamma(Y\times B)=\nu(B)$ for all Borel subsets $A,B \subset Y$. Elements $\gamma \in \Gamma(\mu, \nu)$ are called couplings of the measures $\mu$ and $\nu$, i.e., joint distributions on $Y\times Y$ with marginals $\mu$ and $\nu$ on each axis. $p$-Wasserstein distance is a special case of optimal transport when $X = Y$ and the cost function $c(\tilde y,y)=dist(\tilde y,y)^{p}$.
	
	In the discrete case, the definition \eqref{def:p-Wdistance} has a simple intuitive interpretation: given a $\gamma \in \Gamma(\mu, \nu)$ and any pair of locations $(\tilde y,y)$, the value of $\gamma(\tilde y,y)$ tells us what proportion of $\mu$ mass at $\tilde y$ should be transferred to $y$, in order to reconfigure $\mu$ into $\nu$. Computing the effort of moving a unit of mass from $\tilde y$ to $y$ by $dist(\tilde y,y)^{p}$ yields the interpretation of $W_{p}(\mu, \nu)$ as the minimal effort required to reconfigure $\mu$ mass distribution into that of $\nu$. 
	
	In a practical setting \cite{COTFNT}, referred to as a point cloud, the closed-form solution of $\mu$ and $\nu$ may be unknown, instead only $n$ independent and identically distributed (i.i.d.) samples of $\mu$ and $n$  i.i.d. samples of $\nu$ are available. In further discussion, $n$ refers to the size of the problem. We approximate the probability measures $\mu$ and $\nu$ by empirical distribution functions:
	
	\begin{align}\label{def:empiricalPDF}
	\mu=\frac{1}{n}\sum_{i=1}^{n}\delta_{\tilde y^i} \quad \text{and}\quad  
	\nu=\frac{1}{n}\sum_{j=1}^{n}\delta_{y^j},
	\end{align}
	where $\delta_x$ is the Dirac measure. Any element in $\Gamma(\mu,\nu)$ can clearly be represented by a transition matrix, denoted as  $\gamma=(\gamma_{i,j})_{i,j}$ satisfying:   
	
	\begin{align}
	\label{def:bistochasticity}
		\gamma_{i,j} \geq 0; \quad \quad
	 \forall j, ~\sum_{i=1}^{n}\gamma_{i,j}=\frac{1}{n};
	 \quad \quad  
	 \forall i, ~\sum_{j=1}^{n}\gamma_{i,j}=\frac{1}{n}.
	\end{align}	
	Then $\gamma_{i,j}$ means the mass of $\tilde y^i$ that is transferring to $y^j$.
	
	We denote all matrices in $\bR^{n\times n}$ satisfying \eqref{def:bistochasticity} as $\Gamma^{n}$, then \eqref{def:p-Wdistance} becomes
\begin{align}\label{def:Wdistance_approx}
\hat{W}(f):=\left(\inf_{\gamma \in \Gamma^{n}}\sum_{i,j=1}^{n,n}\, dist(\tilde y^i,y^j)^p \gamma_{i,j}\right)^{1/p}.
\end{align}
\begin{remark}
$\Gamma^{n}$ is in fact the set of $n\times n$ doubly stochastic matrix \cite{sinkhorn1964relationship} divided by $n$. 
\end{remark}

Another practical setting, which is commonly used in fields of computer vision \cite{peleg1989unified,ling2007efficient}, is to compute the Wasserstein distance between two histograms. To compare two grey-scale figures (2D, size $n_0 \times n_0$), we first normalize the grey scale such that the values of cells of each picture sum to one. We denote centers of the cell as $\{y^i\}_{i=1}^{n}$ and $\{\tilde y^i\}_{i=1}^{n}$, then we can use two probability measures to represent the two figures:
\begin{align*}
	\mu=\sum_{i=1}^{n}r_{1,i}\delta_{\tilde y^i} \quad \text{and}\quad  
	\nu=\sum_{j=1}^{n}r_{2,j}\delta_{y^j},
	\end{align*}
	where $r_{1,i},r_{2,j} \ge 0, \forall 1 \le i,j \le n$, $\sum\limits_{i=1}^n r_{1,i}=\sum\limits_{j=1}^n r_{2,j}=1$. The discrete Wasserstein distance \eqref{def:Wdistance_approx} keeps the same form while the transition matrix follows different constraints:
		\begin{align}
	\label{def:bistochasticityEMD}
		\gamma_{i,j} \geq 0; \quad \quad
	 \forall j, ~\sum_{i=1}^{n}\gamma_{i,j}=r_{2,j};
	 \quad \quad  
	 \forall i, ~\sum_{j=1}^{n}\gamma_{i,j}=r_{1,i}.
	\end{align}	

Note that in both settings, the computation of Wasserstein distance is reduced to an LP, i.e.,
\begin{align}\label{LP-OT}
    \begin{array}{rl}
        \min & \sum\limits_{1 \le i,j \le n} C_{i,j}\gamma_{i,j} \\
        \st & \sum\limits_{j=1}^n \gamma_{i,j} = r_{1,i}, \quad \sum\limits_{j=1}^n \gamma_{i,j} = r_{2,i}, \quad \gamma_{i,j} \ge 0,
    \end{array}
\end{align}
where $r^1 \triangleq (r_{1,1},...,r_{1,n})^T$ and $r^2 \triangleq (r_{2,1},...,r_{2,n})^T$ are two probability distributions, and $C_{i,j} = dist(\tilde x^i,x^j)^p$. More generally, we can let $r^1$ and $r^2$ be two nonnegative vectors and $C_{i,j} = C(\tilde{y}^i, y^j)$ be any appropriate transportation cost from $\tilde{y}^i$ to $y^j$, so \eqref{LP-OT} also captures the discrete OT.

However, when the number of particles $n$ becomes large, the number of variables (entries of $\gamma$) scales like $n^2$, which leads to costly computation. Therefore, we will discuss random block coordinate descent methods to keep the computational workload in each iteration reasonable.

\section{Random block coordinate descent for standard LP} \label{sec: converg.} In this section, we first generalize the LP problem \eqref{LP-OT} to a standard LP (see Eq.\eqref{linprog}). Then we propose a random block coordinate descent algorithm for resolution. Its almost sure convergence and linear convergence rate in expectation are analyzed. \rev{Last we introduce an inexact version that is implementable.}

We consider the following standard LP problem:
\begin{align} \label{linprog}
\begin{aligned}
\min_{x \in \bR^N} \quad & c^T x \\
\st \quad & Ax = b,\; x \ge 0,
\end{aligned}
\end{align}
where $A \in \bR^{M \times N}$, $b \in \bR^M$, $c \in \bR^N$, hence $M$ is the number of constraint and $N$ is the total degree of freedom. Assume throughout that $M \le N$. Suppose that $\sN \triangleq \{1,\hdots,N\}$ and denote $\sX \triangleq \{ x \in \bR^N \mid Ax = b, x \ge 0 \}$ as the feasible set. Assume that \eqref{linprog} is finite and has an optimal solution. For any $x \in \sX$ and $\sI \subseteq \sN$, denote
\begin{align}
\label{def: setD}
    \sD(x;\sI) & \triangleq \mbox{argmin}_{d \in \bR^N} \left\{ c^T d \mid x + d \ge 0, Ad = 0, d_i = 0, \forall i \in \sN \setminus \sI \right\}. \\
    \label{def: sd}
    q(x;\sI) & \triangleq \min_{d \in \bR^N} \left\{ c^T d \mid x + d \ge 0, Ad = 0, d_i = 0, \forall i \in \sN \setminus \sI \right\}.
\end{align}
Namely, $\sD(x;\sI)$ is the optimal solution set of the linear program in \eqref{def: setD} and $q(x;\sI)$ is the optimal function value. We have that $q(x;\sI) = c^Td$ for any $d \in \sD(x;\sI)$. Denote $\sX^*$ as the optimal solution set of \eqref{linprog}. Then the following equations hold for any $x \in \sX$:
\begin{align}
    \label{eq:Xstar}
    \sX^* & = x + \sD(x;\sN), \\ 
    \label{eq:qN}
    q(x;\sN) & = c^T x^* - c^T x, \quad \forall x^* \in \sX^*.
\end{align}
{\color{black}
\begin{remark}\label{rmk:negative}
    It is worthy to mention that since $d=0$ follows the conditions in \eqref{def: sd}, $\forall (x,\sI)$, $q(x;\sI)\leq 0$. Furthermore, $ q(x;\sN)\leq q(x;\sI)$.
\end{remark}
}
Consider the block coordinate descent (BCD) method for \eqref{linprog}:
\begin{align}\label{alg: bcd}
\begin{aligned}
& \mbox{find}\; d^k \in \sD(x^k,\sI_k), \\
& x^{k+1} := x^k + d^k,
\end{aligned}
\end{align}
where $\sI_k \subset \sN$ is the working set chosen at iteration $k$. Next, we describe several approaches to select the working set $\sI_k$.
\paragraph{Gauss-Southwell-q rule}
Motivated by the {\it Gauss-Southwell-q} rule introduced in \cite{tseng2009coordinate}, we desire to select $\sI_k$ such that
\begin{align}\label{GSq}
    q(x^k;\sI_k) \le v q(x^k;\sN),
\end{align}
for some constant $v \in (0,1]$. Note that by \eqref{eq:qN}, we have
\begin{align}\label{qkN}
    q(x^k;\sN) = c^T (x^* - x^k),
\end{align}
where $x^*$ is an optimal solution of \eqref{linprog}. Therefore, \eqref{def: sd}-\eqref{qkN} imply that
\begin{align}\notag
c^T d^k  & \le vc^T ( x^* - x^k ) \\
\notag
\overset{\eqref{alg: bcd}}{\implies} c^T (x^{k+1} - x^k) & \le vc^T ( x^* - x^k ) \\
\label{recurs-ineq0}
\implies c^T (x^{k+1} - x^*) & \le (1-v)c^T ( x^k -  x^* ).
\end{align}
\eqref{recurs-ineq0} indicates that the gap of function value decays exponentially with rate $1-v$, as long as we choose $\sI_k$ according to the {Gauss-Southwell-q} rule \eqref{GSq} at each iteration $k$. A trivial choice of $\sI_k$ to satisfy \eqref{GSq} is $\sN$ and $v = 1$. However, this choice results in a potential large-scale subproblem in the BCD method \eqref{alg: bcd}, contradicting the purpose of using BCD. Instead, we should set an upper bound on the cardinality of $\sI_k$, namely, a reasonable batch size to balance the computational effort in each iteration and convergence performance of BCD. Next, we discuss the existence of such an $\sI_k$ given an upper bound $l$ on $|\sI_k|$, which necessitates the following concept.
\begin{definition}\label{def: cr}
A vector $\bar d \in \bR^N$ is conformal to $d \in \bR^N$ if 
\begin{align*}
    \supp(\bar d) \subseteq \supp(d), \; \bar d_i d_i \ge 0, \forall i \in \sN.
\end{align*}
\end{definition}

The following Theorem confirms the existence of such an $\sI_k$ that satisfies \eqref{GSq}, the proof of which follows closely to \cite[Proposition 6.1]{tseng2009block}.
\begin{theorem}\label{thm: existGsq}
Suppose that $\rank(A)+1 \le N$. Given any $x \in \sX$, $l \in \{ \rank(A)+1,\hdots,N  \}$ and $d \in \sD(x;\sN)$. There exist a set $\sI \in \sN$ satisfying $|\sI| \le l$ and a vector $\bar d \in \nul(A)$ conformal to $d$ such that
\begin{align}
    \label{eq: supp}
    \sI & = \supp(\bar d).\\
    \label{GSq1}
    q(x;\sI) & \le \frac{1}{N-l+1}q(x; \sN).
\end{align}
\end{theorem}
\begin{proof}
If $d = 0$, then let $\bar d = 0$ and $\sI = \emptyset$. We have $q(x;\sI) = q(x;\sN) = 0$. Therefore, both \eqref{eq: supp} and \eqref{GSq1} are satisfied. If $d \neq 0$ and $|\supp(d)| \le l$, then let $\bar d = d$. Thus, $\sI = \supp(\bar d)$ satisfies $|\sI| \le l$ and $q(x;\sI) = q(x;\sN)$. 
If $|\supp(d)| > l$, then similar to the discussion in \cite[Proposition 6.1]{tseng2009block}, we have that
\begin{align*}
    d = d^{(1)} + \hdots + d^{(r)},
\end{align*}
for some $r \le |\supp(d)| - l + 1$ and some nonzero $d^{(s)} \in \nul(A)$ conformal to $d$ with $|\supp(d^{(s)})| \le l$, $s = 1,...,r$. Since $|\supp(d)| \le N$, we have $r \le N - l + 1$. Since $Ad^{(s)} = 0$ and $x_i + d^{(s)}_i \ge x_i + d_i \ge 0, \forall s = 1,...,r$ and $\forall i \in \{i \mid d_i < 0\}$, we have that $x + d^{(s)} \in \sX, \forall s = 1,...,r$.  Therefore,
\begin{align*}
    q(x;\sN) & = c^T d = \sum_{s = 1}^r c^T d^{(s)} \ge r \min_{s = 1,\hdots,r} \{ c^T d^{(s)} \}.
\end{align*}
Denote $\bar s \in \mbox{argmin}_{s=1,\hdots,r} \{ c^T d^{(s)} \}$ and let $\sI = \supp(d^{(\bar s)})$, then $|\sI| \le l$ and
\begin{align*}
    q(x;\sN) \ge r c^T d^{(\bar s)} \ge r q(x;\sI) \ge (N-l+1) q(x;\sI),
\end{align*}
{\color{black} where the last  inequalities holds due to $r\leq N-l+1$ and $q(x;\sI)\leq 0$.}

Therefore \eqref{eq: supp} and \eqref{GSq1} hold for this $\sI$ and $\bar d = d^{(\bar s)}$.\qed
\end{proof}

However, it is not clear how to identify the set $\sI$ described in Theorem~\ref{thm: existGsq} with little computational effort for a general $A$. Therefore, we introduced the following.

\paragraph{Expected Gauss-Southwell-q rule}
We introduce randomness in the selection of $\sI_k$ to reduce the potential computation burden in identifying an $\sI_k$ that satisfies \eqref{GSq}. Consider an {\it expected Gauss-Southwell-q rule}:
\begin{align} \label{EGSq}
\bE[ q(x^k;\sI_k) \mid \sF_k ] \le v q(x^k;\sN),
\end{align}
where $v \in (0,1]$ is a constant, and $\sF_k \triangleq \{ x^0, \hdots, x^k \}$ denotes the history of the algorithm. Therefore, using the notations of LP \eqref{linprog} and BCD method \eqref{alg: bcd}:
\begin{align}\notag
\eqref{def: sd}\eqref{qkN}\eqref{EGSq} & \implies \bE[ c^T d^k \mid \sF_k ] \le vc^T ( x^* - x^k )\\
\notag
& \implies \bE[ c^T(x^{k+1} - x^k) \mid \sF_k ] \le vc^T ( x^* - x^k ) \\
\notag
& \implies \bE[c^T (x^{k+1} - x^*) \mid \sF_k] - c^T (x^k - x^*) \le vc^T ( x^* - x^k ) \\
\label{recurs-ineq}
& \implies \bE[c^T (x^{k+1} - x^*) \mid \sF_k] \le (1-v)c^T ( x^k -  x^* ),
\end{align}
where $x^*$ is an optimal solution of \eqref{linprog}. According to \cite[Lemma 10, page 49]{polyak1987introduction}, $c^T(x^k - x^*) \to 0$ almost surely. Moreover, if we take expectations on both sides of \eqref{recurs-ineq},
\begin{align*}
\bE[c^T (x^{k+1} - x^*)] & \le (1-v)\bE[ c^T ( x^k -  x^* ) ] \\
\implies \bE[c^T ( x^k -  x^* )] & \le (1-v)^k \bE[ c^T ( x^0 -  x^* ) ].
\end{align*}
i.e., the expectation of function value gap converges to $0$ exponentially with rate $1-v$. 

\paragraph{Vanilla random block coordinate descent}
Based on the expected Gauss-Southwell-$q$ rule, we formally propose a vanilla random block coordinate descent ({\rbcd}) algorithm (Algorithm~\ref{alg:rbcd}) to solve the LP \eqref{linprog}. Specifically, we choose the working set $\sI_k$ with full randomness, that is, randomly choose an index set of cardinality $l$ out of $\sN$. Then with probability at least $\frac{1}{\binom{N}{l}}$, the index set will be the same as or cover the working set suggested by Theorem~\ref{thm: existGsq}. As a result, \eqref{EGSq} will be satisfied with $v \ge \frac{1}{\binom{N}{l}(N-l+1)}$.

\begin{algorithm}[H]\caption{Vanilla random block coordinate descent ({\rbcd})}\label{alg:rbcd}
\begin{algorithmic}
\STATE {\bf (Initialization)} Choose feasible $x^0 \in \bR^N$ and the batch size $l$ such that $\rank(A) + 1 \le l \le N$.
\FOR{$k=0,1,2,\dotsc$}
\STATE {\bf Step 1. } Choose $\sI_k$ uniformly randomly from $\sN$ with $| \sI_k | = l$.
\STATE{\bf Step 2. } 
Find $
d^k \in \sD(x^k;\sI_k).
$
\STATE{\bf Step 3.}
$
\;\; x^{k+1} := x^k + d^k.
$
\ENDFOR
\end{algorithmic}
\end{algorithm}
Based on the previous discussions, Algorithm~\ref{alg:rbcd} generates a sequence $\{ x^k \}$ such that the value of $c^T x^k$  converges to the optimal with probability 1. Moreover, the expectation of the optimality gap converges to $0$ exponentially. It is important to note that $\frac{1}{\binom{N}{l}(N-l+1)}$ is only a loose lower bound of $v$. This bound can become quite small when $N$ grows large due to the binomial coefficient $\binom{N}{l}$. However, in our numerical experiments (c.f. Sec.~\ref{sec:num}), this lower bound is rarely reached. In Section~\ref{sec:RBCD-OT}, we will discuss how to further improve this bound given the specific structure of the OT problem. \rev{Before that, we first investigate an inexact extension of {\rbcd}.}

\paragraph{\bf Inexact extension of Algorithm~\ref{alg:rbcd}.}
\rev{For any $x$ and $\sI \subseteq \sN$, still denote $\sD(x;\sI)$ and $q(x;\sI)$ as in \eqref{def: setD} and \eqref{def: sd}. However, note that now $x$ does not have to be feasible. We consider the inexact version of Algorithm~\ref{alg:rbcd}, where Step 2 in \eqref{alg:rbcd} is only approximately solved. We compute $d^k$ such that for any $d^k_* \in \sD(x^k; \sI_k)$,
\begin{align}\label{def: dk-inex}
c^T d^k \le (1-\epsilon_k) c^T d^k_*, \quad x^k + d^k \ge 0,\quad  \| A d^k \| \le \hat \epsilon_k, \quad d^k_i = 0, \; \forall i \notin \sI_k.
\end{align}
where the inexactness sequences $\{ \epsilon_k \}$ and $\{ \hat \epsilon_k \}$ should be nonnegative. The inexact algorithm is described as follows.
\renewcommand{\thealgorithm}{1(a)} 
\begin{algorithm}[H]\caption{Inexact random block coordinate descent ({IRBCD})}\label{alg:rbcd-inex}
\begin{algorithmic}
\STATE {\bf (Initialization)} Choose feasible $x^0 \in \bR^N$, the batch size $l$ such that $\rank(A) + 1 \le l \le N$, and inexactness sequences $\{ \epsilon_k \}$ and $\{ \hat \epsilon_k \}$.
\FOR{$k=0,1,2,\dotsc$}
\STATE {\bf Step 1. } Choose $\sI_k$ uniformly randomly from $\sN$ with $| \sI_k | = l$.
\STATE{\bf Step 2. } 
Find $
d^k $ such that \eqref{def: dk-inex} holds.
\STATE{\bf Step 3.}
$
\;\; x^{k+1} := x^k + d^k.
$
\ENDFOR
\end{algorithmic}
\end{algorithm}
Next, we analyze the sequence generated by Algorithm~\ref{alg:rbcd-inex} and summarize the results in the following theorem.
\begin{theorem}\label{thm: inex}
Consider Algorithm~\ref{alg:rbcd-inex}. Given $\delta > 0$, suppose that $0 \le \epsilon_k \le \epsilon < 1$ and $\sum_{k=1}^\infty \he_k \le \delta$. Then we have that 
\be
\item $\| Ax^k - b \| \le \delta$ and $x^k \ge 0$ a.s. for any $k \ge 0$,
\item $c^T x^k \ge c^T x^* - \kappa \delta$ for any $k \ge 0$,
\item $\bE[ c^T x^k - (c^T x^* + \kappa \delta) ] \le (1-(1-\epsilon)v)^k\bE[ c^T x^0 - (c^T x^* + \kappa \delta)  ] $, 
\ee
where $\kappa$ is a constant only depending on the matrix $A$ and vector $c$, and $v$ is the constant in the expected Gauss-Southwell-q rule.
\end{theorem}
\begin{proof}
The proof of 1 is straightforward. Note that for any $k \ge 0$, we have in a.s. sense for any $k \ge 0$,
\begin{align*}
    \| A x^k - b \| = \left\| A x^0 - b + \sum_{t=0}^{k-1} Ad^t \right\| \le \| A x^0 - b \| + \sum_{t=0}^{k-1} \| A d^t \| \overset{\eqref{def: dk-inex}}{\le} \sum_{t=0}^{k-1} \he_t \le \delta.
\end{align*}
$x^k \ge 0$ a.s. is also a direct result of \eqref{def: dk-inex} ($x^k + d^k \ge 0$). Now we prove 2. First we argue that $\sD(x^k; \sN) \neq \emptyset$. This is true by the duality theory in linear programming and the fact that $Ax^k \in S$ and the dual feasibility set $A^T p \le c$ is nonempty. Suppose that $x^{k+1}_* \in \sD(x^k; \sN)$. Then we have
\begin{align*}
    \| A x^{k+1}_* - b \| = \| A x^k - b \| \le \delta.
\end{align*}
We want to estimate $c^T x^{k+1}_*$. Suppose that $p^1,\hdots,p^s$ are all the extreme points of the dual feasible region $\{ p \mid A^T p \le c \}$ and denote $b^k = A x^k$. Then by the theory of LP, we have
\begin{align*}
    c^T x^{k+1}_* = \max_{i = 1,\hdots,s} (b^k)^T p^i  = b^T p^{i^*}
\end{align*}
Note that we also have $c^T x^* = \max_{i=1,\hdots,s} b^T p^i$, then
\begin{align} \label{ineq: objerr}
    c^T x^{k+1}_* - c^T x^*  \le (b^k - b)^T p^{i^*} \le \| b^k - b \| \| p^{i^*} \| \le  \left(\max_{i=1,\hdots,s} \| p^i \| \right) \delta.
\end{align}
Similarly we can show
\begin{align*}
    c^T x^k - c^T x^* \ge c^T x^{k+1}_* - c^T x^* \ge -\left(\max_{i=1,\hdots,s} \| p^i \| \right) \delta,
\end{align*}
which concludes 2 if we let $\kappa \triangleq \max_{i=1,\hdots,s} \| p^i \|$. According to Theorem~\ref{thm: existGsq} and discussion about the exact algorithm, we still have that the expected Gauss-Southwell-q rule \eqref{EGSq} holds with $v \ge \frac{1}{\binom{N}{l}(N-l+1)}$.
Then
\begin{align*}
 & \bE[ c^T d^k \mid \sF_k ] \overset{\eqref{def: dk-inex}}{\le} (1-\epsilon_k)\bE[ c^T d^k_* \mid \sF_k ] \overset{\eqref{EGSq}}{\le} (1-\epsilon_k)vq(x^k;\sN) \\
\implies & \bE[ c^T x^{k+1} - c^T x^k \mid \sF_k] \le (1-\epsilon)v( c^T x^{k+1}_* - c^T x^k ) \\
\overset{\eqref{ineq: objerr}}{\implies} & \bE[ c^T x^{k+1} - c^T x^k \mid \sF_k] \le (1-\epsilon)v( c^T x^* + \kappa \delta - c^T x^k )\\
\implies & \bE[ c^T x^{k+1} - (c^T x^* + \kappa \delta) \mid \sF_k ] \le (1-(1-\epsilon)v)[ c^T x^k - (c^T x^* + \kappa \delta) ], \\
\implies & \bE[ c^T x^k - (c^T x^* + \kappa \delta) ] \le (1-(1-\epsilon)v)^k\bE[ c^T x^0 - (c^T x^* + \kappa \delta)  ], \mbox{ for any } k \ge 0.
\end{align*}\qed
\end{proof}
\begin{remark}
    (i). Theorem~\ref{thm: inex} justifies the inexact algorithm. In particular, if we let $\delta$ be small, then we showed that the expected objective function value sequence also converges to the vicinity of the optimal one linearly. During the implementation, we could either let $\he_k$ to be a sequence proportional to $1/k^{\alpha}, \alpha > 1$, or choose them equally as a sufficiently small number so that their accumulation is also insignificant. Choice of the sequence $\he_k$ is less stringent if we occasionally project the iterates to the feasible region so that the accumulated error is offset. \\
    (ii). In the proof we implicitly assume that the dual feasible region exists at least one extreme point. This can be implied if $A$ has linearly independent rows. In fact, we can eliminate redundant rows of $A$. After this operation, the feasible region remains the same; the analysis and implementation of the algorithms are also similar. \\
    (iii). Inexact versions of the algorithms proposed in Section~\ref{sec:RBCD-OT} are of similar fashion and we will omit detailed and repetitive explanations.
\end{remark}}

\section{Random block coordinate descent and optimal transport}\label{sec:RBCD-OT}

Denote the cost matrix $C \triangleq (C_{i,j})_{i,j}$ in \eqref{LP-OT}. Then calculating the OT between two measures with finite support (problem \eqref{LP-OT}) is a special case of \eqref{linprog}, where $c=\ve(C)$, and $N=n^2$. The constraint matrix $A$ has the following structure:
\begin{align}\label{def: OTcoeff}
    A \triangleq \underbrace{\left(
\begin{array}{llll}
I_n & I_n & \hdots & I_n \\
{\bf 1}_n^T & & & \\
& {\bf 1}_n^T  & & \\
&  &  \ddots & \\
&  &  & {\bf 1}_n^T
\end{array}\right)}_{\text{$n$ blocks}},
\end{align}
where $I_n$ is an $n \times n$ identity matrix, ${\bf 1}_n$ is an $n$ dimensional vector of all $1$'s (then $M = 2n$). Right hand side $b$ in \eqref{linprog} has the form $b \triangleq ((r^1)^T,(r^2)^T)^T$, where $r^1, r^2 \in \bR^n_+$ can be two discrete probability distributions. Next, we discuss the property of matrix $A$ and $\nul(A)$. 

\paragraph{Property of matrix $A$}
A nonzero $d \in \bR^N$ is an {\it elementary vector} of $\nul(A)$ if $d \in \nul(A)$ and there is no nonzero $d' \in \nul(A)$ that is conformal to $d$ and $\supp(d') \neq \supp(d)$. 
According to the definition in \eqref{def: OTcoeff}, we say that a nonzero matrix $X$ is an {\it elementary matrix} of $\nul(A)$ if $\ve(X)$ is an elementary vector of $\nul(A)$. For simplicity, a matrix $M^1$ being conformal to $M^2$ means $\ve(M^1)$ being conformal to $\ve(M^2)$ for the rest of this paper. 

Now we define a set $\EA$:
\begin{align*}
    & X \in \EA \subseteq \bR^{n \times n} \Longleftrightarrow X \neq 0,\; \mbox{and after row and column permutations, X is }\\
    & \mbox{a multiple of one of the following matrices: }\\
    & E^2=\pmat{1 & -1 & & & \\ -1 & 1 & & & \\ && 0 && \\ &&& \ddots & \\ &&&&0}_{n \times n}, 
    E^3=\pmat{1 & & -1 & & & \\ -1 & 1 & & & & \\ & -1 & 1 && & \\ & &  & 0 & & \\ &&& & \ddots & \\ &&&& & 0}_{n \times n},...,\\
    & E^{n-1}=\pmat{1 &  & & & -1 & \\ -1 & 1 & & & & \\ & -1 & 1 && & \\ &&& \ddots && \\ &&& -1 & 1 & \\ &&&&&0}_{n \times n},
     E^{n}=\pmat{1 &  & & & -1 \\ -1 & 1 & & & \\ & -1 & 1 && \\ &&& \ddots & \\ &&& -1 & 1}_{n \times n}.
\end{align*}
First, we state a Lemma about $\EA$, the proof of which is trivial and thus omitted.
\begin{lemma}
Every matrix in $\EA$ is an elementary matrix of $\nul(A)$.
\end{lemma}
{\color{black}
Then we show every element in the null space of $A$ is related to a matrix in $\EA$.
\begin{lemma}\label{lem:cr}
For any nonzero $D$ such that $\ve(D) \in \nul(A)$, there exists $X \in \EA$ such that $X$ is conformal to $D$.
\end{lemma}

    \begin{proof}
        We prove this by contradiction and induction. 
We assume a nonzero $D$ such that $\ve(D) \in \nul(A)$ and no $X \in \EA$ is conformal to $D$. 

Note that $\ve(D) \in \nul(A)$ implies 
\begin{align}\label{vecDinnullA}
    \sum_{i=1}^m D(i,\bar j) = \sum_{j=1}^n D(\bar i,j) = 0, \forall \bar i,\bar j.
\end{align} Then without loss of generality, suppose that $$D(1,1) > 0.$$ Otherwise since $D$ is nonzero we can permute row/column to let $D(1,1) \neq 0$ and by \eqref{vecDinnullA} all entries in the first row/column after permutation cannot have the same sign.

By \eqref{vecDinnullA}, the first column of $D$ must have one negative element. Suppose $D(2,1) < 0$ WLOG. The second row of $D$ must have one positive element, so suppose $D(2,2) > 0$ WLOG. Since no $X \in \EA$ is conformal to $D$, we must have $D(1,2) \ge 0$ otherwise $D$ is conformal to $E^2$. Therefore, the $2\times 2$ principal matrix of $D$ has the following sign arrangement (after appropriate row/column permutations),
\begin{align*}
    \pmat{+ & +/0 \\ - & +},
\end{align*}
where we use $+$, $+/0$, $-$, and $-/0$ to indicate that the corresponding entry is positive, nonnegative, negative, and nonpositive respectively. 

If $n = 2$, then the above pattern is impossible as the first column sums to some positive number, leading to a contradiction with \eqref{vecDinnullA}. Suppose that $n \ge 3$. For math induction, we assume  that,

\noindent $(\mathbf{H}^k):$ after appropriate row/column permutations, the $k \times k$ principal matrix of $D$ has the following sign arrangement ($2 \le k \le n-1$), 
\begin{align}\label{signmatk}
    \pmat{+ & +/0 & +/0 & \hdots & +/0 \\ - & + & +/0 & \ddots & \vdots \\ -/0 & - & + & \ddots & +/0 \\ \vdots & \ddots & \ddots & \ddots & +/0 \\ -/0 & \hdots & -/0 & - & + },
\end{align}
namely, 
\begin{align*}
    \left\{\begin{array}{ll}
    D(i,i) > 0, &  1 \le i \le k; \\
        D(i+1,i) < 0, \quad  &  1 \le i \le k-1;\\
        D(i,j) \ge 0,&  1 \le i \le j \le k; \\
        D(i,j) \le 0, &  1 \le j < i \le k.
    \end{array}\right.
\end{align*}

$k$th column of $D$ needs to have at least one negative element, otherwise, it contradicts with \eqref{vecDinnullA}.  WLOG, we suppose $D(k+1,k)<0$.

We now claim that the rest of the elements in the $(k+1)$th row has to be non-positive, namely $D(k+1,i) \le 0$, $\forall i = 1,...,k-1$. Otherwise, let $i_0$ be the largest index in $\{1,\cdots, k-1\}$ such that $D(k+1,i_0)>0$. Then the submatrix  $D(i_0+1:k+1,i_0:k)$ takes the form,
\begin{align}\label{signmatk-sub}
    \pmat{ - & + & +/0 & \hdots & +/0 \\ -/0 & - & + & \ddots & \vdots \\ \vdots & \ddots & \ddots & \ddots & +/0 \\ -/0 & \hdots & -/0 & - & + \\ + & -/0 &\hdots &-/0 & - },
\end{align}
and it is conformal to $E^{k-i_0+1}$ after row/column permutations. To see this, we can move the first column of \eqref{signmatk-sub} to the last. For the whole matrix $D$, it is equivalent to move the $i_0$th column and insert it between $k$ and $k+1$th column and shift the resulting submatrix to the upper left corner through permutation operations.

While due to \eqref{vecDinnullA}, $(k+1)$th row of $D$ needs to have at least one positive element and we have just shown that $D(k+1,i) \le 0$, $\forall i = 1,...,k$, so WLOG we suppose $D(k+1,k+1)>0$ . Similar argument shows if there is no $X \in \EA$ is conformal to $D$, so $D(i,k+1) \ge 0$, $\forall i = 1,...,k$. 

Therefore, the $(k+1)\times(k+1)$ principal matrix of $D$ has exactly the same sign pattern as indicated by \eqref{signmatk}, after appropriate row/column permutations. So we have the induction $(\mathbf{H}^k)\Rightarrow(\mathbf{H}^{k+1})$.

However, when the $n\times n$ matrix $D$ follows $(\mathbf{H}^{n})$, it has the sign pattern as \eqref{signmatk} after row/column permutations. Then the summation of the last column is strictly positive as indicated by the sign pattern, which contradicts with \eqref{vecDinnullA}.\qed
    \end{proof}
Finally, we show that $\EA$ characterizes all the elementary matrices. 
\begin{theorem}\label{thm: cr}
Given any $D \in \bR^{n \times n}$, if $\ve(D) \in \nul(A)$, then $D$ has a conformal realization \cite[Section 10B]{rockafellar1999network}, namely:
\begin{align}\label{eq: cr}
    D =  D^{(1)} + D^{(2)} + \hdots +  D^{(s)},
\end{align}
where $D^{(1)},\hdots,D^{(s)}$ are elementary matrices of $\nul(A)$ and $D^{(i)}$ is conformal to $D$, for all $i = 1,\hdots,s$. In particular, $D^{(i)} \in \EA$, $\forall i = 1,...,s$. Therefore, $\EA$ includes all the elementary matrices of $\nul(A)$.
\end{theorem}
\begin{proof}
By Lemma \ref{lem:cr}, we suppose that $X^{(1)} \in \EA$ and $X^{(1)}$ is conformal to $D$. Then $X^{(1)}$ can be scaled properly by $\alpha_1=\min_{(i,j)\in \supp(X^{(1)})\cap \supp(D)}|D(i,j)/X^{(1)}(i,j)| > 0$ such that $|\supp(D - \alpha_1 X^{(1)})| < |\supp(D)|$ and $D - \alpha_1 X^{(1)}$ is conformal to $D$. Denote $D^{(1)} \triangleq \alpha_1 X^{(1)}$ and $\bar D^{(1)} = D - D^{(1)}$, both of them are conformal to $D$. 

By apply the same procedure to $\bar D^{(1)}$, we get $\bar D^{(2)}$ such that $|\supp (\bar D^{(2)})|<|\supp (\bar D^{(1)})|$. We can repeat this process and eventually, we have that the conformal realization \eqref{eq: cr} holds since $|\supp(D)| \le n^2$ and $|\supp (\bar D^{(k)})|$ is strictly decreasing. 

If $D$ is an elementary matrix, in previous construction of \eqref{eq: cr}, we notice that $D=D^{(1)}+\bar D^{(1)}$ and both of them are conformal to $D$. If $\bar D^{(1)}\not=0$, we have $\supp(D^{(1)})\not = \supp(D)$ while $D^{(1)}$ is conformal to $D$. This contradicts the assumption of $D$ as an elementary matrix. When $\bar D^{(1)}=0$, $D=D^{(1)}$ is a multiple of the special matrix in the description of $\EA$ after a certain row/column permutation. Summing up, $\EA$ includes all the elementary matrices of $\nul(A)$.
\end{proof}}

\begin{remark}\label{rmk:compare_with_tseng}
For a given $D$ such that $\ve(D) \in \nul(A)$, a simple algorithm following the proof in Theorem~\ref{thm: cr} to find an elementary matrix $X$ conformal to $D \neq 0$ will cost at most $O(n^2)$ operations. We can select an appropriate $\alpha > 0$ such that $D - \alpha X$ is conformal to $D$ and $| \supp(D - \alpha X) | < | \supp(D) |$. By repeating this process, we can find the conformal realization in $\supp(D) \leq n^2$ steps. Therefore, the total number of operations needed to find the conformal realization is $O(n^4)$. In comparison, the approach proposed by \cite{tseng2010coordinate} finds a conformal realization with support cardinality less than $l$ (usually, $l$ is much smaller than $n^2$) and requires $O(n^3(n^2 - l)^2)$ operations.
\end{remark}
 
%

\paragraph{Working set selection} By analyzing the structure of elementary matrices of $\nul(A)$, we will have a better idea of potential directions along which the transport cost is minimized by a large amount. This is supported by the following theorem, where we continue using notations introduced in Section~\ref{sec: converg.}.
\begin{theorem}\label{thm: GSqOT}
Consider the linear program \eqref{linprog} where $A \in \bR^{M \times N}$ and $b \in \bR^M$ are defined as in \eqref{def: OTcoeff} ($M=2n$, $N = n^2$). Given any $X \in \bR^{n \times n}$ and $D \in \bR^{n \times n}$ such that $\ve(X) \in \sX$, and $\ve(D) \in \sD(\ve(X);\sN)$. There exists an elementary matrix $\bar D$ of $\nul(A)$ conformal to $D$ such that for any set $\sI \in \sN$ satisfying
\begin{align*}
    \sI \supseteq \supp(\ve(\bar D)),
\end{align*}
We have 
\begin{align}\label{ineq：GSqOT}
    q(\ve(X);\sI) \le \left( \frac{1}{n^2 - 3} \right) q(\ve(X); \sN).
\end{align}
\end{theorem}
\begin{proof}
Since $\ve(D) \in \sD(\ve(X);\sN)$, $\ve(D) \in \nul(A)$. Then based on Theorem~\ref{thm: cr}, we have the conformal realization:
\begin{align*}
    D = D^{(1)} + D^{(2)} + ... + D^{(s)}.
\end{align*}
Moreover, proof of Theorem~\ref{thm: cr} indicates that we can construct this realization with $s \le n^2 - 3$, because the support of $D^{(i)}$ has cardinality at least $4$. Then similar to discussion in Theorem~\ref{thm: existGsq}, we may find $\bar s \in \{1,\hdots, s\}$ such that $\bar D = D^{(\bar s)}$, $\sI \supseteq \supp(\ve(D^{(\bar s)}))$, and
\begin{align*}
    q(\ve(X);\sN) \ge (n^2 - 3) q(\ve(X);\sI).
\end{align*}\qed
\end{proof}

Now we discuss two approaches to carefully select the support set $\sI_k$ at iteration $k$ of the block coordinate descent method \eqref{alg: bcd}:\\
\be
\item[1.] {\it Diagonal band.} Given $3 \le p \le n$, denote
\begin{align*}
\mathcal{G} \triangleq 
\left\{ (i,j) \in \mathbb{Z}^2 \Big | \begin{array}{ll}
   i \in [j,j+p-1] & \text{if} \; j \in [1,n-p+1] \\
   i \in [1,...,j+p-n-1] \cup [j,n] & \text{if} \; j \in [n-p+2,n] \\
\end{array}
\right\},
\end{align*}
and construct matrix $G \in \bR^{n \times n}$ such that
\begin{align}\label{def: G}
    G(i,j) = \begin{cases}
1, & \text{if} (i,j) \in \mathcal{G},\\
0,  & \text{otherwise.}
\end{cases}
\end{align}
Therefore, $G$ has the following structure:
\begin{align*}
    \begin{array}{r} p \; \left\{ \begin{array}{r} \\ \\ \\ \\
    \end{array} \right.\\
    \begin{array}{r}
        \\ \\ \\ \\
    \end{array}
    \end{array} \hspace{-0.3cm}
    \pmat{
    \;1\; &&&& 1 & \hdots & 1 \\ 
    \vdots & 1\; &&&& \ddots & \vdots \\
    1 & \vdots & \ddots & &&& 1 \\
    1 & 1 && 1\; &&&\\
    & 1 & \ddots & \vdots & 1\; &&\\
    && \ddots & 1 & \vdots & \ddots &\\
    &&& 1 & 1 & \hdots & 1\;
    }_{n \times n}
    \hspace{-1 cm}
    \begin{array}{r} \left. \begin{array}{r} \\ \\ \\
    \end{array} \right\} (p-1) \; \\
    \begin{array}{r}
        \\ \\ \\ \\ \\
    \end{array}
    \end{array}
\end{align*}
It is like a band of width $p$ across the diagonal, hence the name. Then we may construct $\bar D^k \in \bR^{n \times n}$ and $\sI_k$ as follows:
\begin{align}\label{diagbd}
    \begin{aligned}
& \text{Obtain $\bar D^k$ by uniformly randomly permuting all columns and rows of $G$.} \\
& \text{Let} \sI_k \triangleq \supp(\ve(\bar D^k)).
    \end{aligned}
\end{align}
Note that $|\sI_k| = np$.

\item[2.] {\it Submatrix.} Given $m < n$, obtain $\bar D^k$ and $\sI_k$ such that
\begin{align}\label{submat}
    \begin{aligned}
    & \text{Uniformly randomly pick two sets of $m$ different numbers out of $\{1,...n\}$:} \\ 
    & \quad i_1,...,i_m \text{and} j_1,...,j_m. \\
    & \text{Let} \bar D^k(i,j) = 
    \begin{cases}
    1 & \text{if} i \in \{i_1,...,i_m\} \text{and} j \in \{j_1,...,j_m\},\\
    0 & \text{otherwise.}
    \end{cases}
    \\
    & \text{Let} \sI_k \triangleq \supp(\ve(\bar D^k)).
    \end{aligned}
\end{align}
In this case, the support of $\bar D^k$ is a submatrix of size $m \times m$. Therefore, $| \sI_k | = m^2$.
\ee
Next, we discuss two random block coordinate descent algorithms to solve {\color{black}the LP problem \eqref{linprog} with $A$ given in \eqref{def: OTcoeff}} whose working set selections are based on the two approaches discussed above. 

\setcounter{algorithm}{1}
\begin{algorithm}[H]\caption{Random block coordinate descent - diagonal band ({\rbcddb})}\label{alg:rbcddb}
\begin{algorithmic}
\STATE {\bf (Initialization)} Choose feasible $X^0 \in \bR^{n \times n}$ and band width $p \in [3,n]$. Let $x^0 = \ve(X^0)$.
\FOR{$k=0,1,2,\dotsc$}
\STATE {\bf Step 1. } Choose $\sI_k$ according to \eqref{diagbd}.
\STATE{\bf Step 2. } 
Find $
d^k \in \sD(x^k;\sI_k).
$
\STATE{\bf Step 3.}
$
\;\; x^{k+1} := x^k + d^k.
$
\ENDFOR
\end{algorithmic}
\end{algorithm}

The following result describes the convergence property of Algorithm~\ref{alg:rbcddb}.
\begin{theorem}\label{thm: db}
Consider {\color{black}the LP problem \eqref{linprog} with $A$ given in \eqref{def: OTcoeff}}. Then sequence $\{ x^k \}$ and $\{ \sI_k \}$ generated by Algorithm~\ref{alg:rbcddb} satisfies the expected Gauss-Southwell-$q$ rule \eqref{EGSq}, i.e.,
\begin{align*}
    \bE[ q(x^k;\sI_k) \mid \sF_k] \le v q(x^k;\sN),
\end{align*}
with $v \ge \frac{n(p-2)}{(n^2 - 3)(n!)^2}$. Therefore, $c^T(x^k - x^*) \to 0$ almost surely and $\bE[c^T(x^k - x^*)]$ converges to $0$ exponentially with rate $1-v$.
\end{theorem}
\begin{proof}
Given $x^k$, Theorem~\ref{thm: GSqOT} guarantees that there exists $D^k \in \sE_A$ such that if $\sI_k \supseteq \supp(\ve(D^k))$, then \eqref{ineq：GSqOT} holds for $\sI = \sI_k$ and $\ve(X) = x^k$, i.e.,
\begin{align}
    q(x^k;\sI_k) \le \left( \frac{1}{n^2 - 3} \right) q(x^k; \sN).
\end{align}
Next, we will estimate the probability that $\sI_k \supseteq \supp(\ve(D^k))$ holds.

Suppose that after row/column permutations and scaling of $D^k$, we obtain $E^t$, $2 \le t \le n$. Then after appropriate row and column swapping, $D^k$ can be written as
\begin{align}\label{def: Dt}
\begin{array}{r}
    \begin{array}{c}
        \\
    \end{array} \\
    t \left\{ \begin{array}{c}
         \\ \\ \\ \\ \\ \\
    \end{array} \right.   \\
    \begin{array}{c}
         \\ \\
    \end{array} 
\end{array} \hspace{-0.3cm}
    \pmat{ 0\; &&&&&&&&\\ 
    *\; & *\; &&&&&&& \\
    *\; && *\; &&&&&& \\
    & *\; && \ddots\; &&&&& \\
    && *\; && *\; &&&& \\
    &&& \ddots\; && *\; &&& \\
    &&&& *\; & *\; & 0\; && \\
    &&&&&&& \ddots\; & \\
    &&&&&&&& 0\;
    }_{n \times n}.
\end{align}
That is, elements $(2,1)$ and $(3,1)$ are nonzeros; elements $(j,j)$ and $(\mo(j+2,n),j)$ are nonzeros, for all $j = 2,...,t-1$; elements $(t,t)$ and $( \mo(t+1,n),t )$ are nonzeros; all other elements are zeros. Obviously, support of this matrix is covered by the support of $G$ in \eqref{def: G}. Moreover, by moving the whole support in matrix \eqref{def: Dt} downwards or to the bottom right corner, we can create at least $n(p-2)-1$ more different matrices whose support are all covered by $G$. These $n(p-2)$ matrices can be obtained by permuting rows and columns of $D^k$ in $n(p-2)$ in different ways. Therefore, the probability that $\sI_k$ will cover the support of $D^k$ is at least $\frac{n(p-2)}{(n!)^2}$, and we have that
\begin{align*}
    \bE[q(x^k;\sI_k) \mid x^k] & = \sum_{\supp(\ve(D^k)) \subseteq \sI} q(x^k;\sI)P(\sI_k = \sI) + \sum_{\supp(\ve(D^k)) \nsubseteq \sI} q(x^k;\sI)P(\sI_k = \sI) \\
    & \le \left( \frac{1}{n^2 - 3} \right) q(x^k;\sN) P(\supp(\ve(D^k)) \subseteq \sI_k ) +  0 \\
    & \le \left( \frac{n(p-2)}{(n^2 - 3)(n!)^2} \right) q(x^k;\sN)
\end{align*}
Therefore, the expected Gauss-Southwell-$q$ rule \eqref{EGSq} holds with $v$ at least $\frac{n(p-2)}{(n^2 - 3)(n!)^2}$.\qed
\end{proof}
\begin{remark}
It can be shown that if $n$ is large enough and $p$ is chosen between $O(log(n))$ and $O(n)$, then the lower bound for constant $v$ derived in Theorem~\ref{thm: db} is better than the one estimated for Algorithm~\ref{alg:rbcd}, i.e., $\frac{1}{\binom{N}{l}(N - l +1)}$. In fact,
we have the following results.
\begin{lemma}\label{lm: prange}
Suppose that $ \bar K \ge 2 $ and $\eta > 0$ satisfies 
\begin{align*}
     \frac{ 2 \bar K - 3 }{2 (\bar K - 1)} + \log \left(\frac{\bar K}{2} \right) > 2 / \eta,
\end{align*}
and $n$ satisfies
\begin{align*}
    n \ge \frac{4}{ \left( \frac{ 2 \bar K - 3 }{2 (\bar K - 1)} + \log \left(\frac{\bar K}{2} \right) \right) \eta - 2}, \quad \frac{n}{\log(n)} \ge \eta \bar K.
\end{align*}
Then for any $p \in [\eta \log(n), \frac{n}{\bar K}]$, and $p \ge 3$, we have $\frac{n(p-2)}{(n^2 - 3)(n!)^2} \ge \frac{1}{\binom{n^2}{np}(n^2 - np +1)}$.
\end{lemma}
\begin{proof}
See Appendix.\qed
\end{proof}
\noindent Let $n \ge 30$, $\eta = 1$, $\bar K = 8$. Then according to Lemma~\ref{lm: prange}, for $log(n) \le p \le n/8$, the lower bound $\frac{n(p-2)}{(n^2 - 3)(n!)^2}$ is larger. We believe that this is a fairly reasonable range of $p$ when $n$ grows large. This lower bound is improved because we have knowledge of the structure of the elementary matrix when solving OT problems.
\end{remark}
As for the submatrix approach, we often find it quite efficient in numerical experiments. The optimal solution in the submatrix approach can also be decomposed in $\EA$. More precisely, consider that $\sI_k$ is chosen according to \eqref{submat} associated with a submatrix of size $m \times m$. 
Then, the decomposition of $d^k \in \sD(x^k;\sI_k)$ (more rigorously its matrix form) only involves multiples and permutations of the elementary matrices in $\{E^2,E^3,\cdots, E^m\}$.
However, global convergence with a fixed-width submatrix is not guaranteed. In fact, there is a counterexample (see \ref{app: counterex}). Therefore, we propose an algorithm that combines these two approaches together.

\begin{algorithm}[H]\caption{Random block coordinate descent - submatrix and diagonal Band ({\rbcdsdb})}\label{alg:rbcdsdb}
\begin{algorithmic}
\STATE {\bf (Initialization)} Choose feasible $X^0 \in \bR^{n \times n}$, submatrix row/column dimension $m$, band width $p \in [3,n]$ and selection parameter $s \in [0,1]$. Let $x^0 = \ve(X^0)$.
\FOR{$k=0,1,2,\dotsc$}
\STATE {\bf Step 1. } With probability $s$, choose $\sI_k$ according to \eqref{diagbd}; otherwise, choose $\sI_k$ according to \eqref{submat}.
\STATE{\bf Step 2. } 
Find $
d^k \in \sD(x^k;\sI_k).
$
\STATE{\bf Step 3.}
$
\;\; x^{k+1} := x^k + d^k.
$
\ENDFOR
\end{algorithmic}
\end{algorithm}

The convergence of Algorithm~\ref{alg:rbcdsdb} is guaranteed by the next theorem.
\begin{theorem}
Consider {the LP problem \color{black}\eqref{linprog} with $A$ given in \eqref{def: OTcoeff}}. Then sequence $\{ x^k \}$ and $\{ \sI_k \}$ generated by Algorithm~\ref{alg:rbcdsdb} satisfies the expected Gauss-Southwell-$q$ rule \eqref{EGSq}, with $v \ge \frac{sn(p-2)}{(n^2 - 3)(n!)^2}$. Therefore, $c^T(x^k - x^*) \to 0$ almost surely and $\bE[c^T(x^k - x^*)]$ converges to $0$ exponentially with rate $1-v$.
\end{theorem}
\begin{proof}
Given $x^k$, Theorem~\ref{thm: GSqOT} shows that there exists $D^k \in \sE_A$ such that if $\sI_k \supseteq \supp(\ve(D^k))$, then \eqref{ineq：GSqOT} holds with $\sI = \sI_k$ and $\ve(X) = x^k$. We estimate the probability that $\sI_k \supseteq \supp(\ve(D^k))$.

First, consider the case that after row/column permutations and scaling of $D^k$, we obtain $E^t$, $2 \le t \le m$. If $\sI_k$ is chosen according to \eqref{diagbd}, then similar to discussion in Theorem~\ref{thm: db}, $\sI_k$ will cover the support of $D^k$ with probability at least $ \frac{n(p-2)}{(n!)^2} $. If $\sI_k$ is chosen according to \eqref{submat}, then $\sI_k$ will cover the support of $D^k$ with probability
\begin{align*}
    \frac{ \binom{n-t}{m-t}^2 }{ \binom{n}{m}^2 } = \left( \frac{(n-t)!/((m-t)!(n-m)!)}{n!/(m!(n-m)!)} \right)^2 = \left( \frac{m!/(m-t)!}{n!/(n-t)!} \right)^2.
\end{align*}
Therefore, in this case, the probability $p_t$ that $\sI_k$ cover the support of $D^k$ is:
\begin{align*}
    p_t \ge \frac{sn(p-2)}{(n!)^2} + (1-s)\left( \frac{m!/(m-t)!}{n!/(n-t)!} \right)^2.
\end{align*}
Then we consider the case that when we get $E^t$, $m+1 \le t \le n$ after row/column permutations and rescaling of $D^k$. In this case, if $\sI_k$ is chosen according to \eqref{diagbd}, $\sI_k$ will cover the support of $D^k$ with probability at least $ \frac{n(p-2)}{(n!)^2} $; if $\sI_k$ is chosen according to \eqref{submat}, this probability is $0$. Therefore, in this case we have $p_t \ge \frac{sn(p-2)}{(n!)^2}$. In general, the probability that $\sI_k$ cover the support of $D^k$ is at least $\min\{ p_t \} \ge \frac{sn(p-2)}{(n!)^2}$. Similar to discussion in Theorem~\ref{thm: GSqOT}, \eqref{EGSq} will hold with $v \ge \frac{sn(p-2)}{(n^2 - 3)(n!)^2}$.\qed
\end{proof}

\paragraph{Accelerated random block coordinate descent} Algorithm~\ref{alg:arbcd} is an accelerated random block coordinate descent ({\arbcd}) algorithm. It selects the working set $\sI_k$ in a different way from Algorithm~\ref{alg:rbcdsdb} intermittently for acceleration purposes. At certain times, we construct $\sI_k$ based on the iterates generated by the algorithm in the past, i.e., $x^{end} - x^{start}$. This vector reflects the progress achieved by running the {\rbcdsdb} for a few iterations and predicts the direction in which the algorithm could potentially make further improvements. This choice is analogous to the momentum concept often employed in acceleration techniques in optimization, such as in the heavy ball method and Nesterov acceleration. Algorithm~\ref{alg:arbcd} has a similar convergence rate as Algorithm~\ref{alg:rbcdsdb} (note that the acceleration iteration occurs occasionally). We will verify its improved performance in the numerical experiments.

\begin{algorithm}[H]
\caption{Accelerated random block coordinate descent ({\arbcd})}\label{alg:arbcd}
\begin{algorithmic}
\STATE {\bf (Initialization)} Choose feasible $X^0 \in \bR^{n \times n}$, submatrix row/column dimension $m$, band width $p \in [3, n]$, selection parameter $s \in [0, 1]$, and acceleration interval $T$. Let $x^0 = \ve(X^0)$, $x^{start} = x^{end} = x^0$. Binary variable $acc$.
\FOR{$k=0,1,2,\hdots$}
\STATE{\bf Step 1. } Choose $\sI_k$ as following.
\IF{$\mo(k+1,T) \neq 0$ or $|\supp(x^{end} - x^{start})| \le m^2$}
\STATE $acc = \texttt{false}$. With probability $s$, choose $\sI_k$ according to \eqref{diagbd}; otherwise, choose $\sI_k$ according to \eqref{submat}.
\ELSE
\STATE $acc = \texttt{true}$. Choose $\sI_k$ uniformly randomly from $\supp(x^{end} - x^{start})$ so that $|\sI_k| = m^2$.
\ENDIF
\STATE{\bf Step 2. } 
Find $
d^k \in \sD(x^k;\sI_k).
$
\STATE{\bf Step 3. } Update 
$
x^{k+1} := x^k + d^k;
$
\STATE{\bf Step 4. } Update $x^{end} = x^{k+1}.$
\IF {$acc = \texttt{true}$.}
\STATE Update $ x^{start} = x^{k+1}.$
\ENDIF
\ENDFOR
\end{algorithmic}
\end{algorithm}

\section{Numerical experiments}\label{sec:num}
In this section, we conduct numerical experiments on various examples of OT problems\footnote{All experiments are conducted using Matlab R2021b on Dell OptiPlex 7090 with CPU: Intel(R) Core(TM) i9-10900 @ 2.80GHz (20 CPUs), $\sim$2.8GHz and RAM: 65536Mb. Data and codes are uploaded to \texttt{https://github.com/yue-xie/RBCDforOT}.}. In Section~\ref{subsec: numerics-rbcds}, we compare various random block coordinate descent methods with different working set selections proposed in this article. Then, we compare the one with the best performance - {\arbcd} with {\sinkhorn} in Section~\ref{subsec: sink} and an interior point inspired algorithm in Section~\ref{subsec: ipm}. Finally, a large-scale OT problem is solved using {\arbcd} in Section~\ref{subsec: large}.

\subsection{Comparison between various random block coordinate descent methods with the different working set selection rules}\label{subsec: numerics-rbcds}
In this subsection, we apply the proposed random block coordinate descent methods (Alg.~\ref{alg:rbcd} - Alg.~\ref{alg:arbcd}) to calculate the Wasserstein distance between three pairs of distributions. We compare these algorithms to illustrate the difference between various working set selection rules. Additionally, we inspect the differences between the theoretical and actual convergence rates, as well as the solution sparsity.

\paragraph{Experiment settings} We compute the Wasserstein distance between a pair of 1-dim probability distributions (standard normal and uniform over $[-1,1]$), a pair of 2-dim probability distributions (uniform over $[-\pi,\pi]^2$ and an empirical invariant measure obtained from IPM simulation of reaction-diffusion particles in advection flows, detailed configurations can be found in \cite{wang2022deepparticle}, Section 4.2, 2D cellular flow, $\kappa=2^{-4}$), and a pair of 3-dim distributions (uniform over $[-1,1]^3$ and 3-dimensional multivariate normal distribution). When computing the Wasserstein distance between the pair of 1-dim probability distributions, we utilize their histograms (c.f. Section~\ref{sec:WdistanceOT}): Let $n = 1001$. Centers of the cells are $y^i = \tilde y^i = \frac{i-501}{500}$, $i = 1,..,1001$; $C_{i,j} = dist(\tilde{y}^i, y^j)^2$, $1 \le i,j \le 1001$; $r_{1,i} = \frac{\phi(y^i)}{\sum_{i=1}^{1001} \phi(y^i)}$, $i = 1,...,1001$, where $\phi(y)$ is the pdf of standard normal; $ r^2 = (1/1001,...,1/1001)^T \in \bR^{1001}$.  When calculating the Wasserstein distance between the second and third pairs, we apply the point cloud setting (c.f. Section~\ref{sec:WdistanceOT}): Let $n = 1000$. For each pair, use i.i.d. samples $\{ \tilde{y}^i \}$ and $\{ y^j \}$, $1\le i, j \le 1000$ to approximate the two continuous probability measure respectively. Let $C_{i,j} = dist(\tilde{y}^i, y^j)^2$ and $r^1 = r^2 = (1/1000,...,1/1000)^T \in \bR^{1000}$. \rev{In all cases, we normalize the cost matrix $C$ such that its maximal element is $1$.} Figure~\ref{fig: distr.} captures these three pairs of distributions. For all cases, we first use the {\texttt linprog} in {\texttt Matlab} to find a solution with high precision (dual-simplex, constraint tolerance 1e-9, optimality tolerance 1e-10). 
  
\begin{figure}[ht!]
    \begin{center}
    \includegraphics[width = .3\textwidth]{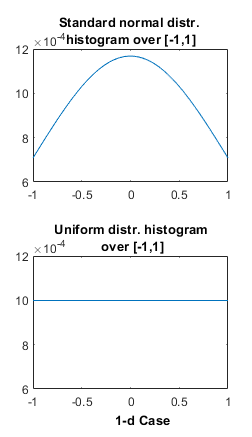}
    \includegraphics[width = .3\textwidth]{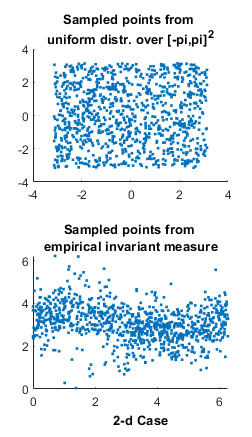}
    \includegraphics[width = .3\textwidth]{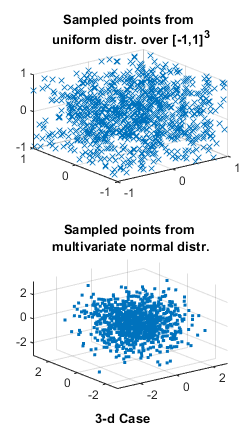}
    \end{center}
    \caption{Three pairs of distributions}
    \footnotesize
    In 1-d case, we compare the histograms of two distributions; in 2-d and 3-d settings, we compare the samples/point clouds of two distributions.
    \label{fig: distr.}
\end{figure}

\paragraph{Methods} We specify the settings of the four algorithms.  All algorithms are started at the same feasible $x^0 = \ve(r^1(r^2)^T)$ in each experiment. We solve the LP subproblems via {\texttt linprog} in {\texttt Matlab} with high precision (dual-simplex, constraint tolerance 1e-9, optimality tolerance 1e-7).\\
\noindent{{\rbcd}.} Algorithm~\ref{alg:rbcd}: Vanilla random block coordinate descent.  Let $l = 100^2$.  Stop the algorithm after $5000$ iterations. \\
\noindent{{\rbcddb}.} Algorithm~\ref{alg:rbcddb}: Random block coordinate descent - diagonal band.  Let $p = \lfloor 100^2/n \rfloor$.  Stop the algorithm after $5000$ iterations. \\
\noindent{{\rbcdsdb}.} Algorithm~\ref{alg:rbcdsdb}: Random block coordinate descent - submatrix and diagonal band. Let $m = 100$, $p = \lfloor m^2/n \rfloor$ and $s = 0.1$. Stop the algorithm after $5000$ iterations. \\
\noindent{\bf ARBCD.} Algorithm~\ref{alg:arbcd}: Accelerated random block coordinate descent. Let $m = 100$, $p = \lfloor m^2/n \rfloor$, $s = 0.1$ and $T = 10$. Stop the algorithm after $5000$ iterations. Note that the degree of freedom of the subproblem per iteration is $100^2$, about $1/100$ the size of the original one.\\

\begin{figure}[ht!]
    \begin{center}
    \includegraphics[width = .3\textwidth]{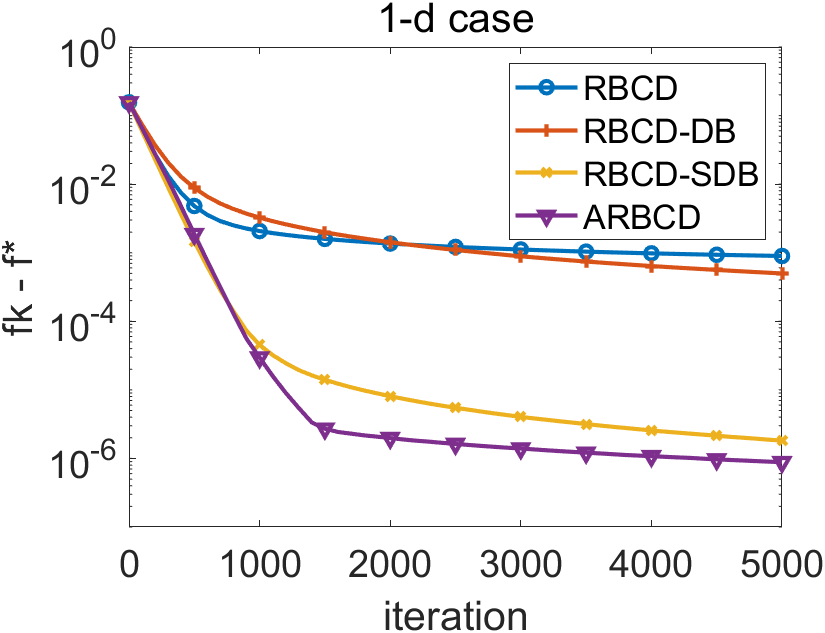}
    \includegraphics[width = .3\textwidth]{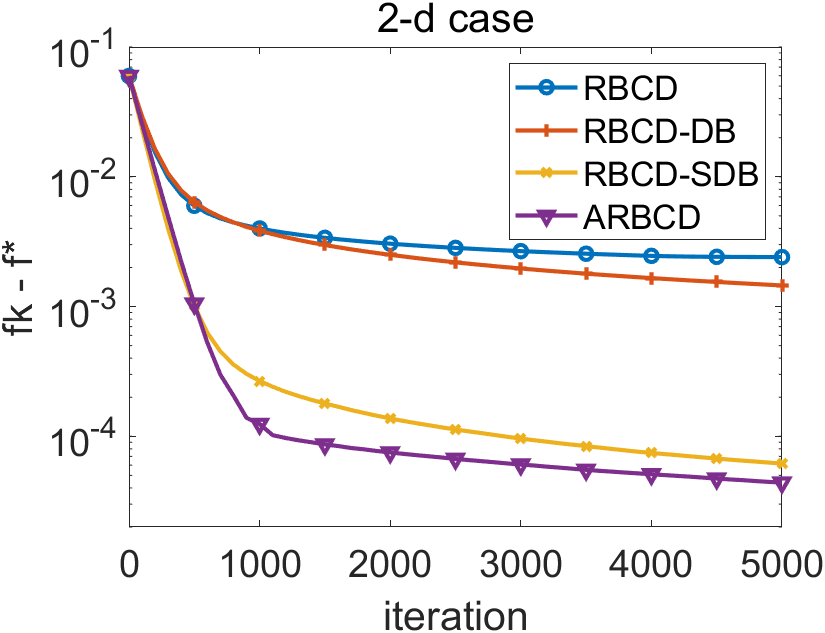}
    \includegraphics[width = .3\textwidth]{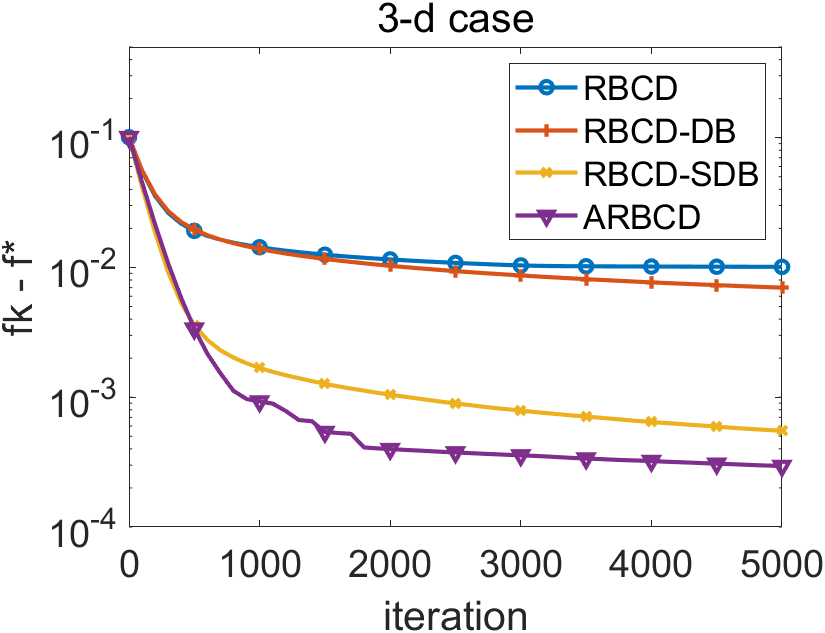}
    \includegraphics[width = .3\textwidth]{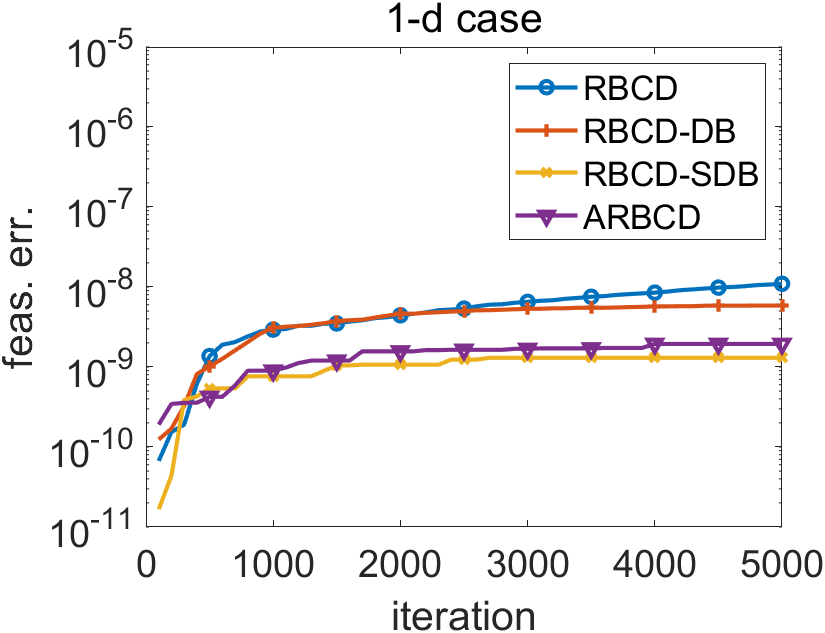}
    \includegraphics[width = .3\textwidth]{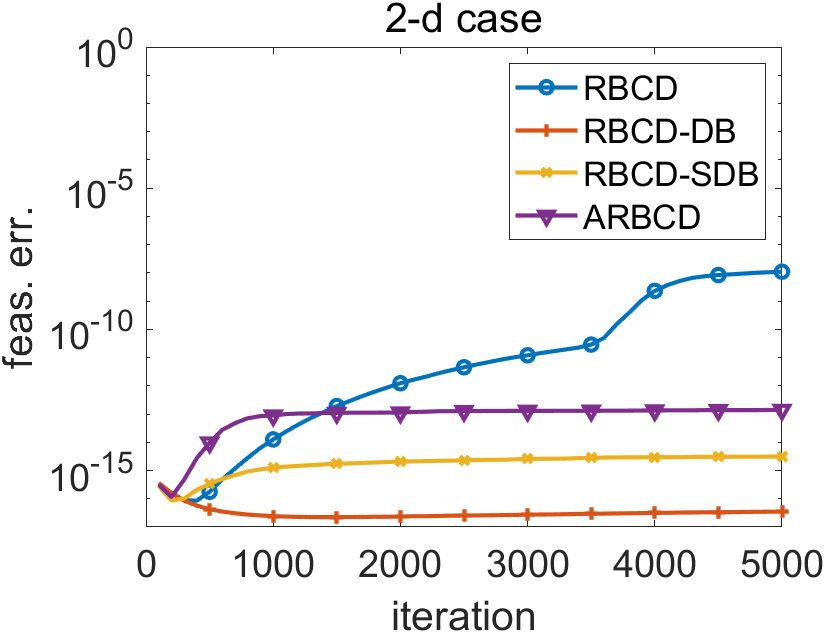}
    \includegraphics[width = .3\textwidth]{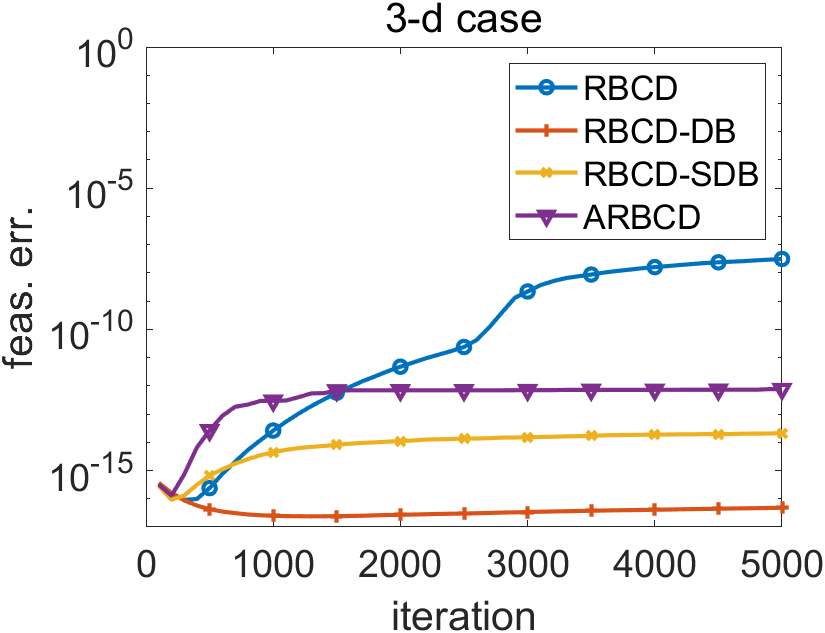}
    \end{center}
    \caption{Comparison of algorithms to compute Wasserstein distance I}
    \footnotesize
    X-axis is the number of operated iterations. Y-axis is the optimality gap $f_k - f^* = c^T x^k - c^T x^*$ (first row) or corresponding feasibility error $\| X^k \cdot {\bf 1} - r^1 \| + \| (X^k)^T \cdot {\bf 1} - r^2 \|$ (second row).  First row of subplots in this figure shows the trajectory/progress of Alg.~\ref{alg:rbcd}, Alg.~\ref{alg:rbcddb}, Alg.~\ref{alg:rbcdsdb} and Alg.~\ref{alg:arbcd} when computing the Wasserstein distance between the three pairs of prob. in 1-d, 2-d, and 3-d respectively. Each algorithm is run 5 times and the curves showcase the average behavior. \rev{Second row of subplots shows the corresponding average feasibility error against computation time for the algorithms in this experiment.}
    \label{fig: traj}
\end{figure}

\paragraph{Comments on Figure~\ref{fig: traj}} We can see from Figure~\ref{fig: traj} that different approaches to choosing the working set of the same size can significantly affect the performance of random BCD types of methods. The curves of {\rbcddb} are below those of {\rbcd} in the long run, demonstrating that {\rbcddb} has a better average performance. The reason for this is that {\rbcd} generates the working set with full randomness, while {\rbcddb} takes the structure of the elementary matrices into account. The latter makes an educated guess at the working set that decreases the objective function by a large amount. The submatrix approach \eqref{submat} works very well in practice, as illustrated by the better performances of {\rbcdsdb} and {\arbcd} compared to {\rbcddb}. In the long run, {\arbcd} outperforms {\rbcdsdb}, verifying the acceleration effect. It is important to note that the algorithm settings are set by default. We expect and have observed similar behaviors of the algorithms when changing their algorithm settings. \rev{Also note that the feasibility error of algorithms are controlled at a low level according to the figure.} On the other hand, the curves in these numerical experiments suggest sublinear convergence rates. This observation does not contradict the theoretical linear convergence rate as long as $v$ is small enough. We will verify that the numerical experiments do not violate the lower bounds we derived for the constant $v$ in the linear convergence rates.

\begin{table}
\footnotesize
  \centering
  \begin{tabular}{ccc|ccc|ccc}
    \toprule
   \multicolumn{3}{c}{ 1-d case } & \multicolumn{3}{c}{ 2-d case} & \multicolumn{3}{c}{ 3-d case} \\
   \midrule
  iter. & $(\bar f_k - f^*)C_{\max}$ & $\hat v$ &  iter. & $(\bar f_k - f^*)C_{\max}$ & $\hat v$ & iter. & $(\bar f_k - f^*)C_{\max}$ & $\hat v$ \\
  \midrule
  0 & 0.6237 & N/A & 0 & 9.3538 & N/A & 0 & 3.3456 & N/A \\ 
   1000 & 0.0083 & 4.3e-3 & 1000 & 0.6254 & 2.7e-3 & 1000 & 0.4746 & 2.0e-3 \\ 
    2000 & 0.0054 & 4.3e-4 & 2000 & 0.4773 & 2.7e-4 & 2000 & 0.3821 & 2.2e-4 \\ 
     3000 & 0.0045 & 1.8e-4 & 3000 & 0.4183 & 1.3e-4 & 3000 & 0.3438 & 1.1e-4 \\ 
      4000 & 0.0039 & 1.4e-4 & 4000 & 0.3846 & 8.3e-5 & 4000 & 0.3371 & 2.0e-5 \\ 
       5000 & 0.0036 & 8.0e-5 & 5000 & 0.3762 & 2.2e-5 & 5000 & 0.3350 & 6.2e-6 \\ 
    \bottomrule
  \end{tabular}
\caption{Data of {\rbcd}, $C_{\max}$ is the maximal value of elements in $C$} \label{tab: RBCD}
\end{table} 

\begin{table}
\footnotesize
  \centering
  \begin{tabular}{ccc|ccc|ccc}
    \toprule
   \multicolumn{3}{c}{ 1-d case } & \multicolumn{3}{c}{ 2-d case} & \multicolumn{3}{c}{ 3-d case} \\
   \midrule
  iter. & $(\bar f_k - f^*)C_{\max}$ & $\hat v$ &  iter. & $(\bar f_k - f^*)C_{\max}$ & $\hat v$ & iter. & $(\bar f_k - f^*)C_{\max}$ & $\hat v$ \\
  \midrule
  0 & 0.6237 & N/A & 0 & 9.3538 & N/A & 0 & 3.3456 & N/A \\ 
   1000 & 0.0130 & 3.9e-3 & 1000 & 0.6002 & 2.7e-3 & 1000 & 0.4601 & 2.0e-3 \\ 
    2000 & 0.0057 & 8.2e-4 & 2000 & 0.3920 & 4.3e-4 & 2000 & 0.3414 & 3.0e-4 \\ 
     3000 & 0.0036 & 4.6e-4 & 3000 & 0.3074 & 2.4e-4 & 3000 & 0.2878 & 1.7e-4 \\ 
      4000 & 0.0026 & 3.3e-4 & 4000 & 0.2592 & 1.7e-4 & 4000 & 0.2544 & 1.2e-4\\ 
       5000 & 0.0020 & 2.6e-4 & 5000 & 0.2276 & 1.3e-4 & 5000 & 0.2316 & 9.4e-5 \\ 
    \bottomrule
  \end{tabular}
    \caption{Data of {\rbcddb}}  \label{tab: RBCD-DB}
\end{table} 

\paragraph{About Table~\ref{tab: RBCD} \& \ref{tab: RBCD-DB}} In these two tables we record the optimality gap $\bar f_k - f^*$ every 1000 iterations for both {\rbcd} and {\rbcddb}. $\bar f_k$ is the average function value at iteration $k$, as we run the algorithms repeatedly for 5 times. The column $\hat v$ denotes the estimation of the constant $v$ in the expected Gauss-Southwell-$q$ rule \eqref{EGSq}. It is calculated by the formula: $ \hat v = 1 - \sqrt[1000]{(\bar f_k - f^*)/( \bar f_{k-1000} - f^* ) } $. Values of $\hat v$ in both Table~\ref{tab: RBCD} \& \ref{tab: RBCD-DB} are far larger than the lower bounds for $v$: $\frac{1}{\binom{N}{l}(N-l+1)}$ and $\frac{(n-1)(p-2)}{(n^2-3)(n!)^2}$, corresponding to {\rbcd} and {\rbcddb} respectively, where $N = n^2$. They also decrease as we run more iterations, indicating that the optimality gap shrinkage becomes less when the iterate is closer to the solution. We intend to study this phenomenon in our future work.

\begin{figure}[ht!]
    \begin{center}
    \includegraphics[width = .3\textwidth]{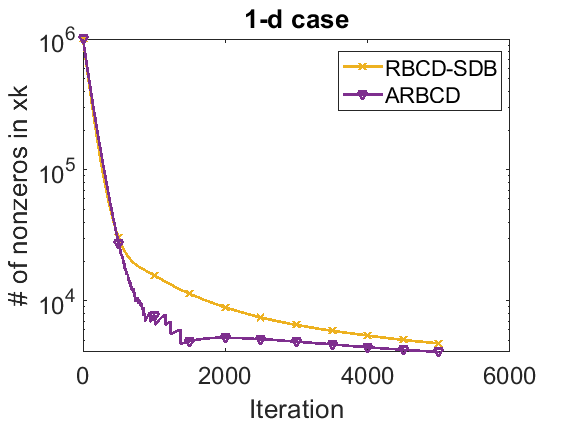}
    \includegraphics[width = .3\textwidth]{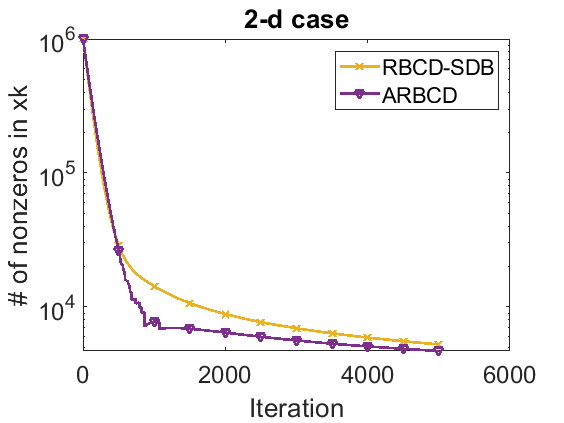}
    \includegraphics[width = .3\textwidth]{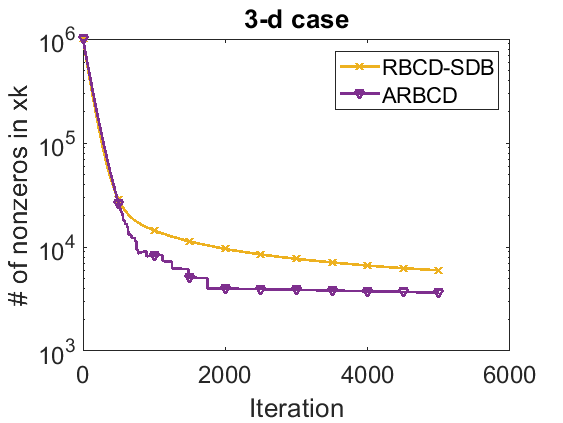}
    \end{center}
    \caption{Sparsity of solutions}
    \footnotesize
    Y-axis records $\| x_k \|_0$, i.e., the number of nonzero elements in $x_k$.  This figure shows the sparsity of $x_k$ in {\rbcdsdb} and {\arbcd} when computing the Wasserstein distance given the three pairs of probability distributions. Each curve represents the average over 5 repetitions.
    \label{fig: card}
\end{figure}

\paragraph{Sparse solutions} We can observe from Figure~\ref{fig: card} that the iterates in {\rbcdsdb} and {\arbcd} become sparse quickly and remain so. The reason for this is that the solutions of OT problems are usually sparse (for the point cloud setting, at least one of the optimal solutions satisfies $\| x^* \| = n$ because extreme points of the LP in this setting are permutation matrices divided by $n$), and these two algorithms can locate solutions with high accuracy relatively fast. As a result, the storage need for these two algorithms is considerably reduced after they have been run for a while. In the point cloud setting, storage complexity is typically expected to decrease from $O(n^2)$ to $O(n)$ (note that the degree of freedom of the subproblem per iteration is typically chosen as $O(n)$ because of the diagonal band approach with $p \ge 3$). 

\paragraph{\rev{Choice of $m$}} \rev{We run {\rbcdsdb} and {\arbcd} with different $m$ for a fixed $n$ to test the optimal setting of subproblem size. As is shown in Figure~\ref{fig: m}, the best choice of $m$ happens at a smaller value (10 or 30 percent of $n$) in 1-d case. In Figure~\ref{fig: m2}, we observe similar phenomenon in 2-d case. The optimal setting of $m$ shifts to a larger value in 3-d case. We conclude that the optimal setting of $m$ may vary depending on many factors such as specific problem and subproblem solver efficiency. Based on the given results, we suggest choosing a smaller $m$ in practice, especially when there is limitation of computation resources.}

\begin{figure}[ht!]
    \begin{center}
    \includegraphics[width = .3\textwidth]{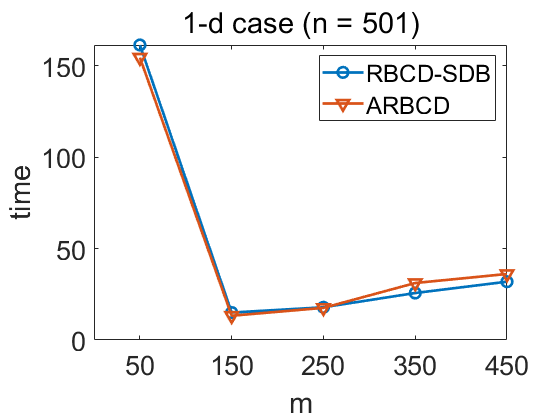}
    \includegraphics[width = .3\textwidth]{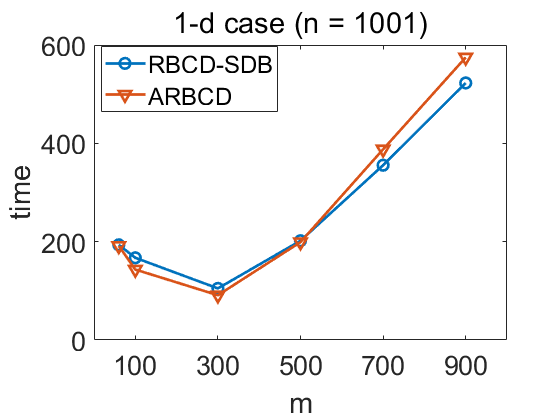}
    \includegraphics[width = .3\textwidth]{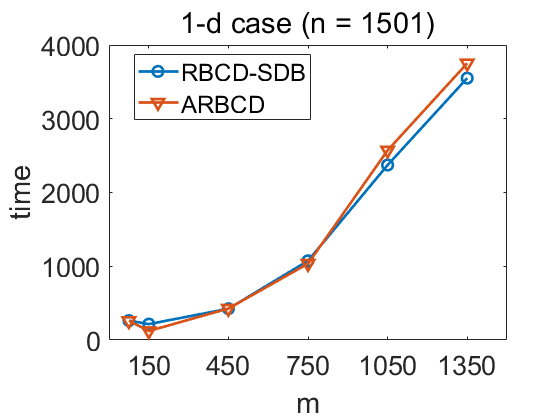}
    \end{center}
    \caption{Choice of $m$ under various $n$}
    \footnotesize
    Wall-clock time of algorithm {\rbcdsdb} and {\arbcd} to compute solutions of accuracy level $f^* \times 10^{-3}$. Algorithms are stopped if the solution accuracy is within the tolerance. Repetition is 3 and the average time is reported. $p = \lfloor m^2/n \rfloor$.
    \label{fig: m}
\end{figure}

\begin{table}[ht!]
\footnotesize
  \begin{center}
  \begin{tabular}{cc|cccccc}
    \toprule
  \multicolumn{8}{c}{ $n = 501$ } \\
  \midrule
   & $m$ & & 50 & 150 & 250 & 350 & 450 \\
  {\rbcdsdb} & iter. & & 2.00e+04 & 3.55e+02 & 55.0 & 17.7 & 7.67 \\
  & err. & & 7.28e-07 & 4.95e-07 & 4.88e-07 & 4.01e-07 & 1.71e-07 \\
  {\arbcd} & iter. & & 1.95e+4 & 3.01e+02& 58.0 & 22.3 & 8.33 \\
  & err. & & 5.68e-07 & 4.90e-07 & 4.92e-07 & 9.08e-08 & 9.49e-08 \\
  \midrule
  \multicolumn{8}{c}{ $n = 1001$ } \\
  \midrule
   & $m$ & 60 & 100 & 300 & 500 & 700 & 900 \\
  {\rbcdsdb} & iter. & 2.00e+04 & 1.14e+04 & 2.21e+02 & 47.7 & 20 & 7.67 \\
  & err. & 3.55e-06 & 4.93e-07 & 4.93e-07 & 2.96e-07 & 3.51e-07 & 1.65e-07 \\
  {\arbcd} & iter. & 2.00e+04 & 9.67e+03 & 1.74e+02 & 47.3 & 22.3 & 8.33 \\
  & err. & 2.06e-06 & 4.94e-07 & 4.60e-07 & 4.09e-07 & 3.39e-07 & 1.17e-08 \\
  \midrule
  \multicolumn{8}{c}{ $n = 1501$ } \\
  \midrule
   & $m$ & 75 & 150 & 450 & 750 & 1050 & 1350 \\
  {\rbcdsdb} & iter. & 2.00e+04 & 6.80e+03 & 1.69e+02 & 50.3 & 19.3 & 7.67 \\
  & err. & 3.90e-06 & 4.93e-07 & 4.89e-07 & 4.18e-07 & 3.63e-07 & 1.70e-07 \\
  {\arbcd} & iter. & 2.00e+04 & 3.11e+03 & 1.67e+02 & 46.3 & 20 & 8 \\
  & err. & 1.85e-06 & 4.93e-07 & 4.86e-07 & 4.15e-07 & 3.08e-07 & 2.35e-07 \\
    \bottomrule
  \end{tabular}
  \end{center}
\caption{Data of Figure~\ref{fig: m}}
The average iterations and solution errors are recorded for reference. Feasibility errors are kept at a low level (below 1e-15) and omitted.
\label{tab: m}
\end{table} 

\begin{figure}[ht!]
    \begin{center}
    \includegraphics[width = \textwidth]{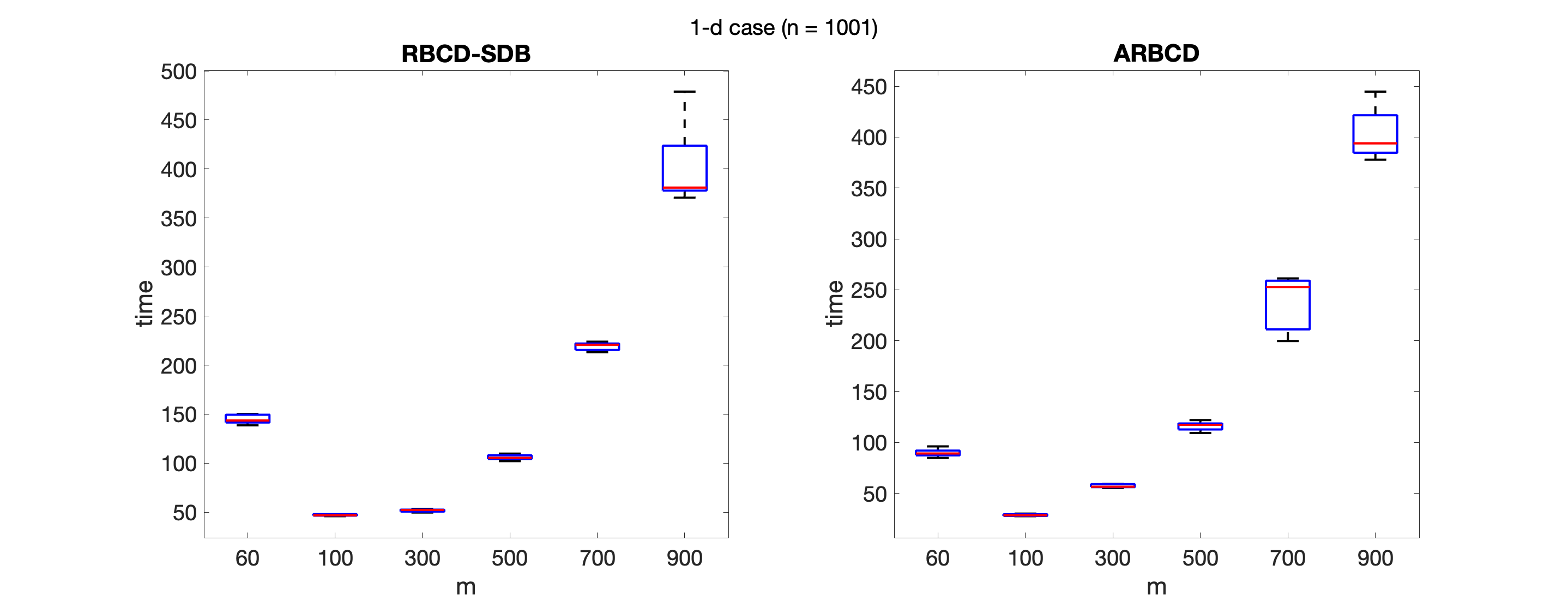}
    \includegraphics[width = \textwidth]{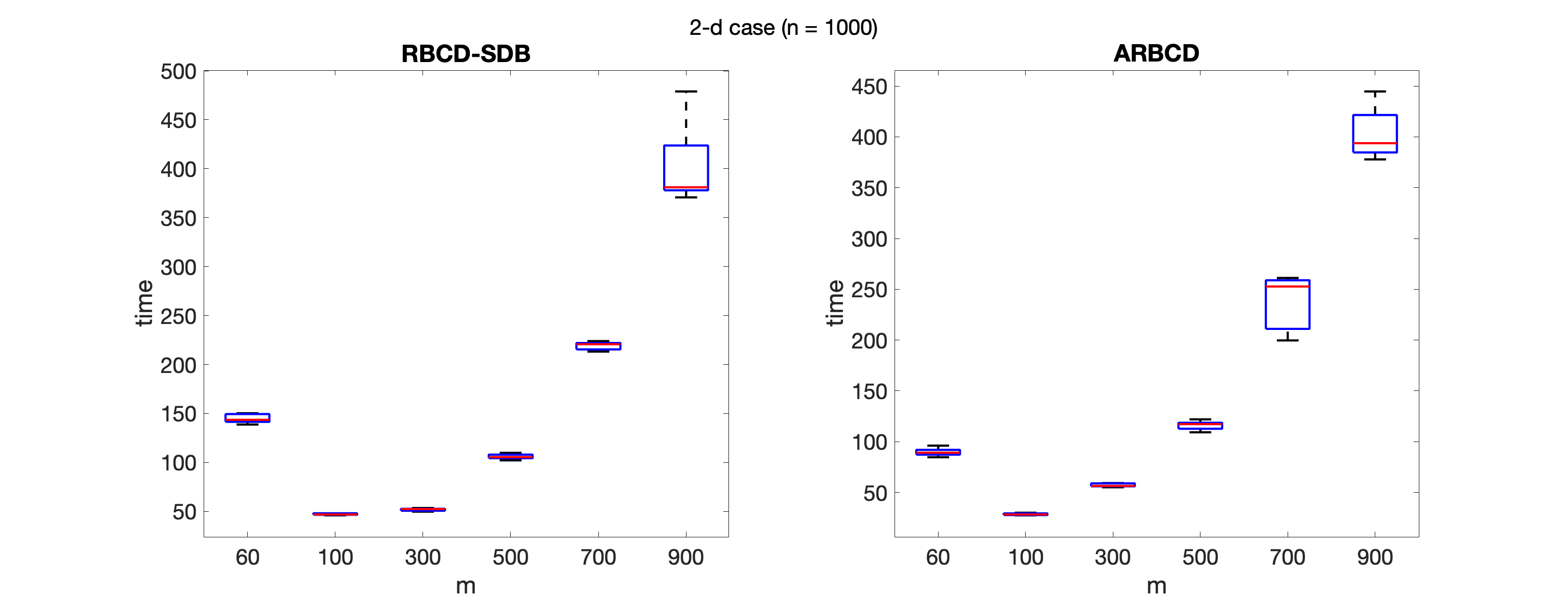}
    \includegraphics[width = \textwidth]{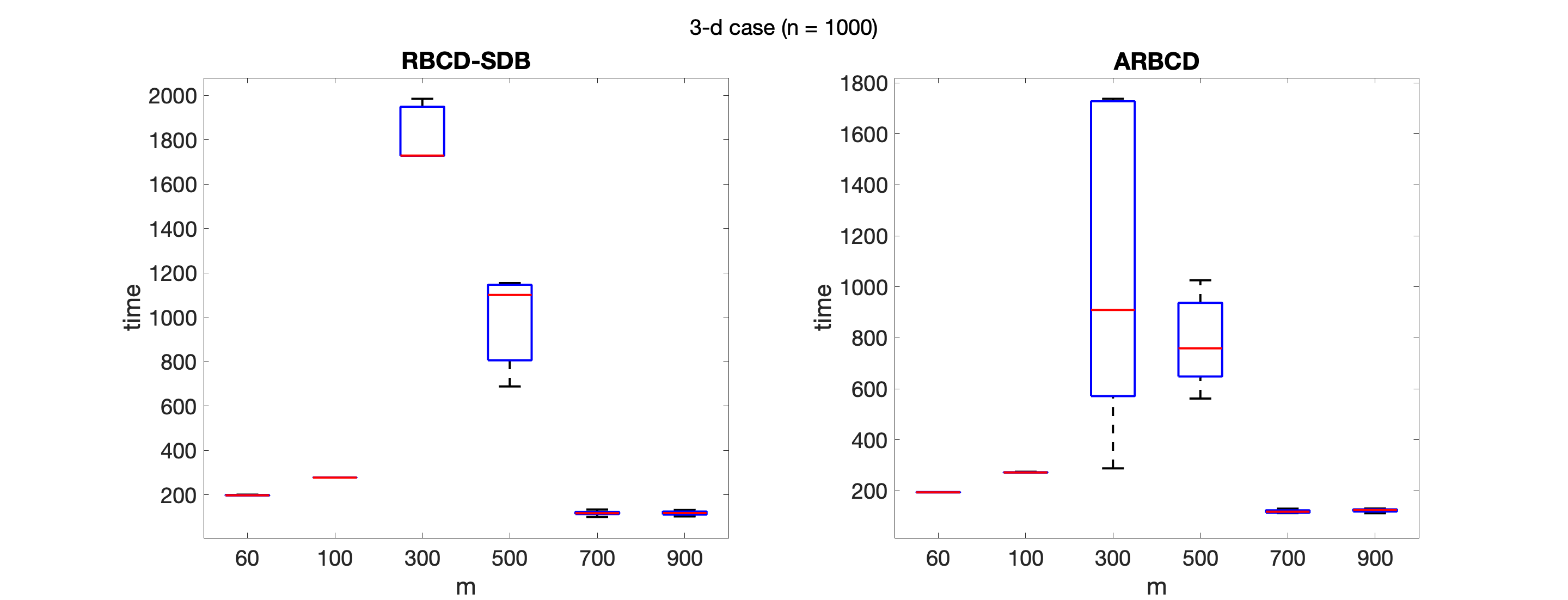}
    \end{center}
    \caption{Choice of $m$ in various dimensions}
    \footnotesize
    Numerical experiments from 1-d to 3-d cases. Solutions of accuracy level $f^* \times 10^{-3}$. Repetition is 5 and the box plot of time is shown. $p = \lfloor m^2/n \rfloor$.
    \label{fig: m2}
\end{figure}

\subsection{Comparison between {\arbcd} and \rev{algorithms in the literature}}\label{subsec: sink}

\paragraph{Experiment settings} We generated 8 pairs of distributions/patterns based on synthetic and real datasets. Descriptions are as follows. Note that we use histogram settings (c.f. Section~\ref{sec:WdistanceOT}) for datasets 1 and 2, and point cloud settings (c.f. Section~\ref{sec:WdistanceOT}) for other datasets. \rev{We use cost function $C(x,y) = \| x - y \|^2$.}\\
{\bf Dataset 1}: Uniform distribution to standard normal distribution over $[-1,1]$. Similar to the 1-d case in Section~\ref{subsec: numerics-rbcds}. $n = 200.$ \\
{\bf Dataset 2}: Uniform distribution to a randomly shuffled standard normal distribution over $[-1,1]$.\footnote{Similar to the 1-d case in Section~\ref{subsec: numerics-rbcds}. We randomly shuffled the weights of the normal distribution histogram.} $n = 1000.$\\
{\bf Dataset 3}: Uniform distribution over $[-\pi,\pi]^2$ to an empirical invariant measure. Similar to the 2-d case in Section~\ref{subsec: numerics-rbcds}. $n = 1000.$\\
{\bf Dataset 4}: Distribution of $\sqrt{\Sigma}u$ to distribution of $2\sqrt{\Sigma}v - (1;1;1)$, where $\Sigma = \pmat{1 & 0.5 & 0.25 \\ 0.5 & 1 & 0.5 \\ 0.25 & 0.5 & 1 }$, $u$ and $v$ conform uniform distributions on $[0,1]^3$ and are independent. $n = 1000.$\\
{\bf Dataset 5}: Similar to Dataset 4, with $\Sigma = \pmat{1 & 0.8 & 0.64 \\ 0.8 & 1 & 0.8 \\ 0.64 & 0.8 & 1 }$. $n = 1000.$\\
{\bf Dataset 6}: Distribution of $\Sigma u$ to distribution of $\Sigma v$, where $\Sigma = \pmat{ 1 & 0 & 1 & 1 \\ 0 & 1 & 1 & -1 }^T$, $u$ conforms a uniform distribution on $[0,2\pi]^2$ and $v$ conforms a uniform distribution on $[-1,1]^2$. $n = 1000.$\\
{\bf Dataset 7}: Distribution of $\underbrace{(1;1;\hdots;1)^T}_{10} u$ to distribution of $(1;2;3;\hdots;10)^T v + (1;1;\hdots;1)^T$, where $u$ conforms uniform distribution over $[0,2\pi]$ and $v$ conforms uniform distribution over $[-1,1]$. $n = 1000.$\\
{\bf Dataset 8}: Distibution of a ``cylinder" to a "spiral", see Figure~\ref{fig: datasets}. $n = 1000.$\\
In all cases, we normalize the cost matrix $C$ such that its maximal element is $1$. For all cases, we use the {\texttt linprog} in {\texttt Matlab} to find a solution with high precision (dual-simplex, constraint tolerance 1e-9, optimality tolerance 1e-10). 

\begin{figure}[ht!]
    \begin{center}
    \includegraphics[width = .5\textwidth]{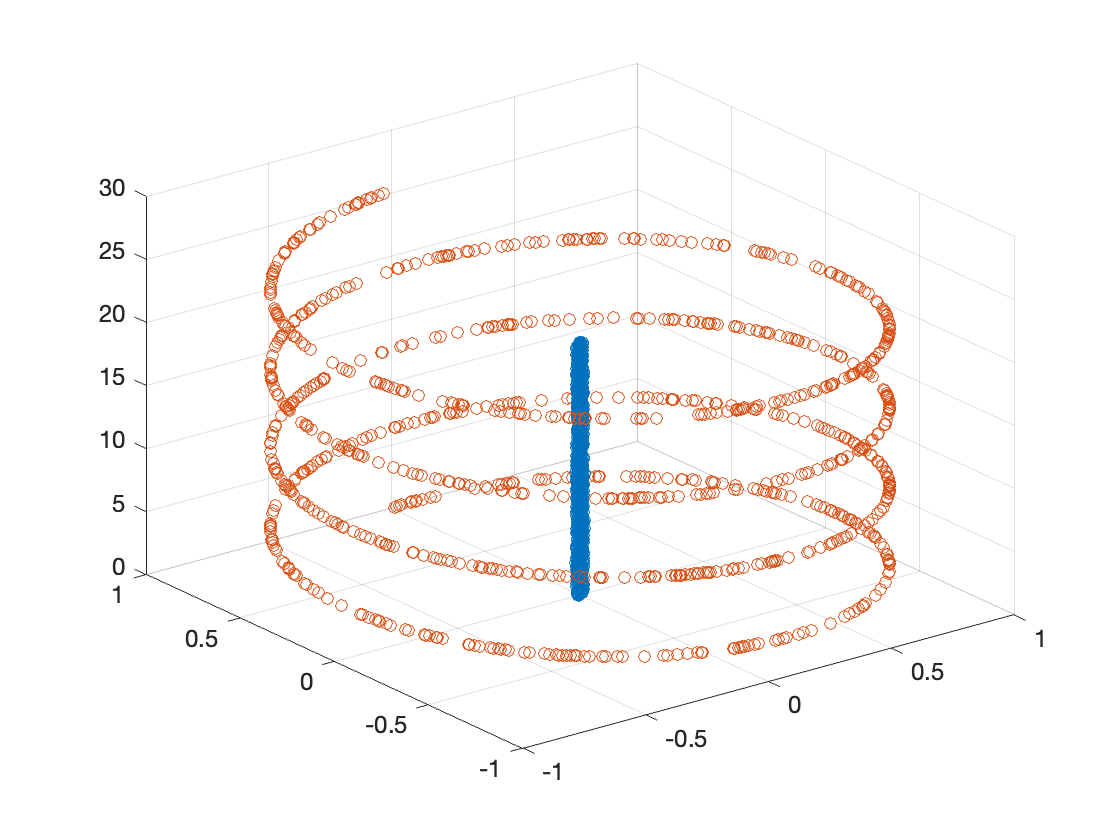}
    \end{center}
    \caption{Visualization of dataset 8}
    \footnotesize
    \label{fig: datasets}
\end{figure}

\subsubsection{Comparison with \sinkhorn}\label{subsec: sink}
\paragraph{Methods} Implementation of {\sinkhorn} and {\arbcd} are specified as follows.\\
\noindent{{\sinkhorn}.} The algorithm proposed in \cite{cuturi2013sinkhorn} to compute Wasserstein distance. Let $\gamma$ be the coefficient of the entropy term. We let $\gamma = \epsilon/(4 \log n)$ as suggested in \cite{dvurechensky2018computational}. We consider the settings $\epsilon = 10^{-4}, 10^{-3}, 0.01, 0.1$.  Iterations of {\sinkhorn} are projected onto the feasible region using a rounding procedure: Algorithm 2 in \cite{altschuler2017near}. Note that this projection step is added only for evaluation purposes because Sinkhorn does not provide feasible solutions if early stopped. It does not affect Sinkhorn's main steps or Sinkhorn's convergence at all. A similar approach is used for evaluation in \cite{jambulapati2019direct}. In addition, we take all the updates to log space and use the LogSumExp function to avoid numerical instability issues. We stop {\sinkhorn} after 300000 iterations when $n = 200$ and 100000 iterations if $n = 1000$.\\
\noindent{{\arbcd}.} Algorithm~\ref{alg:arbcd}: Accelerated random block coordinate descent. Let $m = 40$ when $n = 200$ and $m = 100$ when $n = 1000$. Let $p = \lfloor m^2/n \rfloor$, $s = 0.1$ and $T = 10$. Stop the algorithm after $10000$ iterations. To be fair, we also project the solution in each iteration onto the feasible region via the rounding procedure. LP subproblems are solved via {\texttt linprog} in {\texttt Matlab} with high precision (dual-simplex, constraint tolerance 1e-9).

\begin{figure}[ht!]
    \begin{center}
    \includegraphics[width = .3\textwidth]{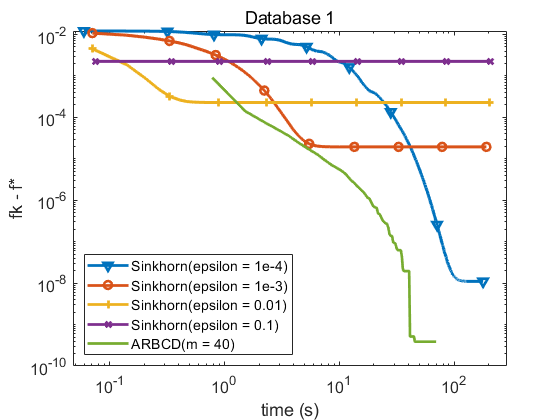}
    \includegraphics[width = .3\textwidth]{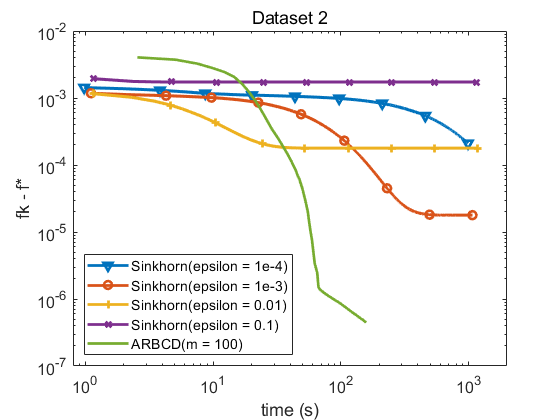}
    \includegraphics[width = .3\textwidth]{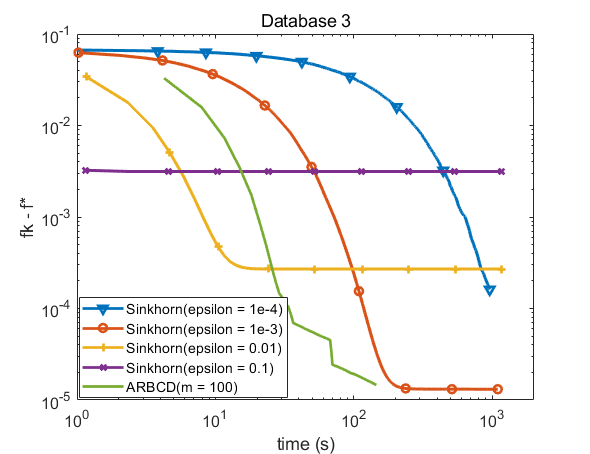}
    \includegraphics[width = .3\textwidth]{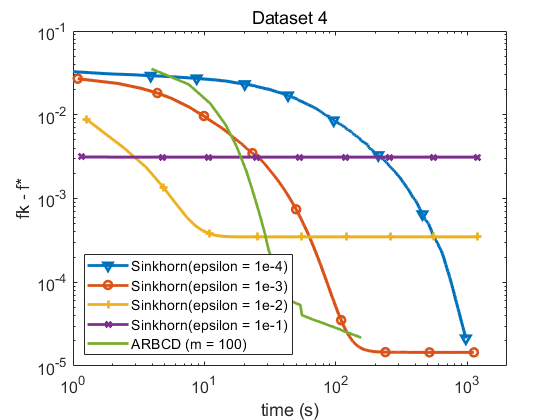}
    \includegraphics[width = .3\textwidth]{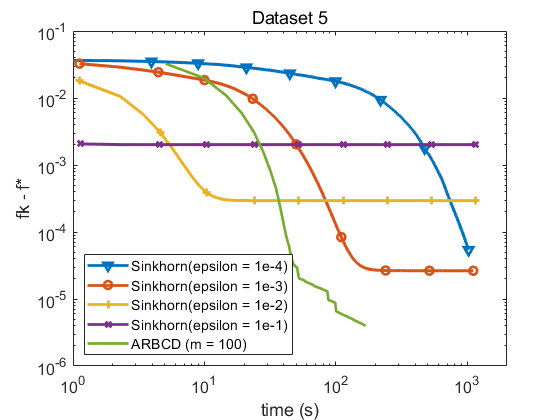}
    \includegraphics[width = .3\textwidth]{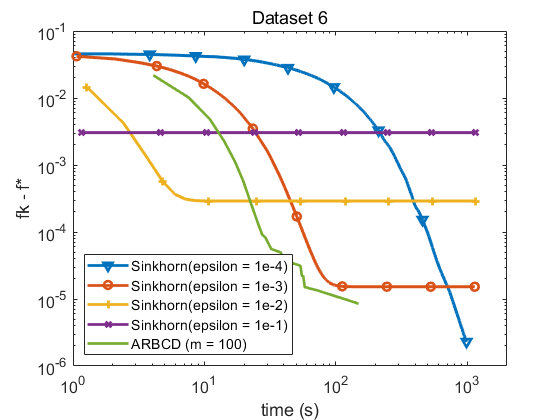}
    \includegraphics[width = .3\textwidth]{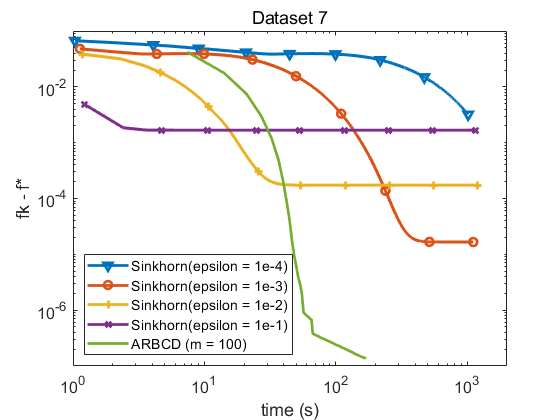}
    \includegraphics[width = .3\textwidth]{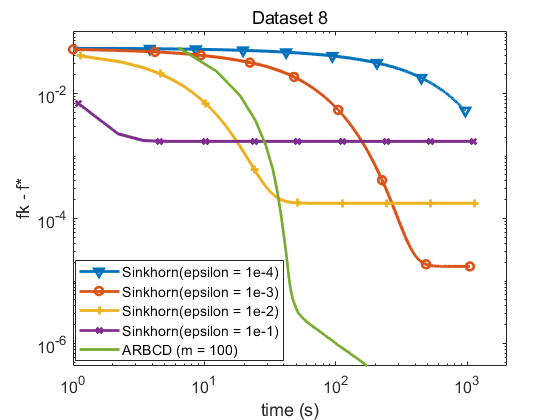}
    \end{center}
    \caption{Comparison of algorithms to compute Wasserstein distance II}
    \footnotesize
    X-axis is the wall-clock time in seconds. Y-axis is the optimality gap $f_k - f^* = c^T x^k - c^T x^*$.  This figure shows the trajectory/progress of Algorithm~\ref{alg:arbcd}: {\arbcd} and {\sinkhorn} with different settings when computing the Wasserstein distance between eight pairs of probability. {\arbcd} is run 5 times in each experiment and the curves showcase the average behavior.
    \label{fig: traj2}
\end{figure}

\begin{figure}[ht!]
    \begin{center}
    \includegraphics[width = 1\textwidth]{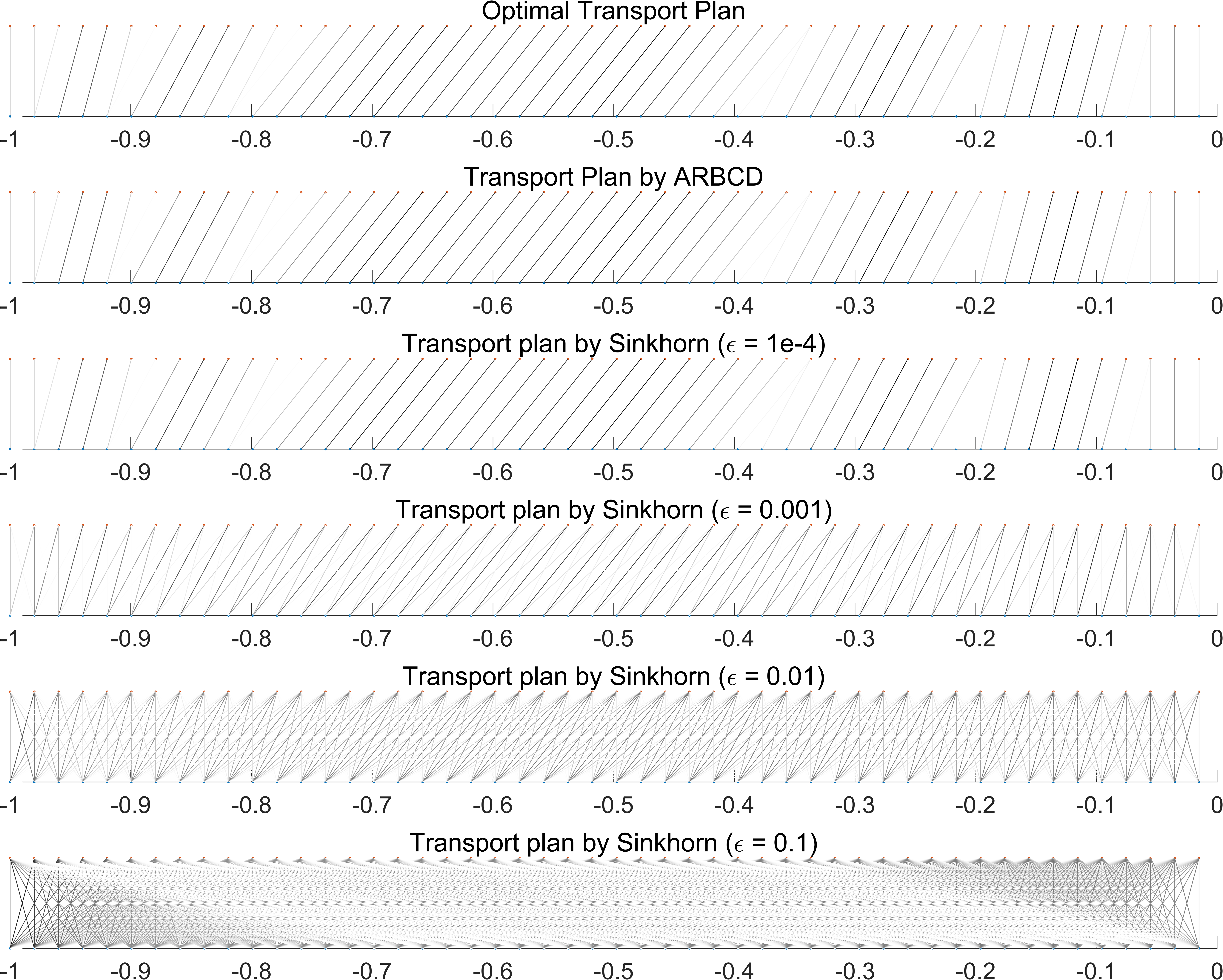}
    \end{center}
    \caption{Transport plan given by different algorithms (dataset 1)}
    \footnotesize
    This figure shows the transport plan for dataset 1. In each plot, the bottom distribution is uniform and the top is standard normal. Each line segment in between represents mass transported between a pair of points. The darker the line is, the more mass is transported. To plot the plans more clearly, we select every other mass point from $-1$ to $0$ (so only include $50$ mass points). The overall transport plans from $-1$ to $1$ are symmetric plots so we show them in this way due to presentation clarity.
    \label{fig: optplan}
\end{figure}

\begin{table}[ht!]
\footnotesize
  \begin{center}
  \begin{tabular}{ccccccc}
    \toprule
  Dataset $\#$ & method & time(s) & iter. & gap & feas.err. & subproblem size 
  \\
  \hline
   & \arbcd & 82.63 & 2326.7 & 4.9386e-07 & 4.1753e-17 & 22500 \\
  & \ipm$_1$ & 7.539 & 198  & 3.36713-08 & 1.8924e-17 & 20890 \\
  1 & \ipm$_2$ & 342.6 & 2000 & 4.3408-04 & 2.0889e-17 & 96550 \\
  & \ipm$_3$ & 5.044 & 19 & 1.2319-07 & 5.7525e-17 & 485194 \\
  & \ipm$_4$ & 11.01 & 21 & 2.8982e-07 & 8.3340e-17 & 1000000 \\
  \hline
  & \arbcd & 240.3 & 10000 & 8.6011e-09 & 8.4518e-17 & 22500 \\
  & \ipm$_1$ & 413.9 & 2000  & 1.0562-04 & 2.5704e-17 & 20890 \\
  2 & \ipm$_2$ & 0.8815 & 16 & 1.1066-10 & 2.6584e-17 & 96550 \\
  & \ipm$_3$ & 5.923 & 20 & 1.7153-09 & 5.5638e-17 & 485194 \\
  & \ipm$_4$ & 11.02 & 20 & 3.0871e-09 & 7.6026e-17 & 1000000 \\
  \hline
  & \arbcd & 30.80 & 375 & 1.2787e-04 & 2.3256e-17 & 22500 \\
  & \ipm$_1$ & N/A & N/A  & N/A & N/A & 21055 \\
  3 & \ipm$_2$ & 67.43 & 2000 & 0.0146 & 2.1107e-17 & 92666 \\
  & \ipm$_3$ & 313.7 & 2000 & 0.0110 & 1.3510e-17 & 465075 \\
  & \ipm$_4$ & 6.696 & 13 & 1.2655e-04 & 6.3473e-17 & 1000000 \\
  \hline
  & \arbcd & 145.2 & 4404.3 & 2.0519e-05 & 1.8927e-17 & 22500 \\
  & \ipm$_1$ & N/A & N/A  & N/A & N/A & 21312 \\
  4 & \ipm$_2$ & N/A & N/A  & N/A & N/A & 95639 \\
  & \ipm$_3$ & 345.7 & 2000 & 0.0083 & 1.5114e-17 & 474798 \\
  & \ipm$_4$ & 7.716 & 15 & 2.0397e-05 & 7.2654e-17 & 1000000 \\
  \hline
  & \arbcd & 56.24 & 569.3 & 2.0412e-05 & 2.1927e-17 & 22500 \\
  & \ipm$_1$ & N/A & N/A  & N/A & N/A & 19988 \\
  5 & \ipm$_2$ & N/A & N/A  & N/A & N/A & 94468 \\
  & \ipm$_3$ & 320.6 & 2000 & 0.0050 & 1.4802e-17 & 492856 \\
  & \ipm$_4$ & 7.616 & 15 & 8.6357e-06 & 6.6523e-17 & 1000000 \\
  \hline
  & \arbcd & 22.20 & 265 & 2.3475e-04 & 2.3622e-17 & 22500 \\
  & \ipm$_1$ & N/A & N/A  & N/A & N/A & 20143 \\
  6 & \ipm$_2$ & 417.6 & 2000 & 0.0032 & 1.4814e-17 & 94072 \\
  & \ipm$_3$ & 337.9 & 2000 & 0.0065 & 1.3566e-17 & 482269 \\
  & \ipm$_4$ & 5.438 & 11 & 1.4044e-04 & 6.2801e-17 & 1000000 \\
  \hline
  & \arbcd & 63.88 & 342.3 & 7.4215e-05 & 2.0341e-17 & 22500 \\
  & \ipm$_1$ & N/A & N/A  & N/A & N/A & 20139 \\
  7 & \ipm$_2$ & N/A & N/A & N/A & N/A & 99792 \\
  & \ipm$_3$ & 278.4 & 2000 & 0.0221 & 1.9256e-17 & 486651 \\
  & \ipm$_4$ & 5.136 & 11 & 5.7172e-05 & 5.5542e-17 & 1000000 \\
  \hline
  & \arbcd & 71.44 & 439.3 & 1.9715e-05 & 2.4507e-17 & 22500 \\
  & \ipm$_1$ & N/A & N/A  & N/A & N/A & 10702 \\
  8 & \ipm$_2$ & N/A & N/A  & N/A & N/A & 97305 \\
  & \ipm$_3$ & 219.1 & 2000 & 9.3439e-04 & 2.3205e-17 & 499786 \\
  & \ipm$_4$ & 5.303 & 10 & 1.7165e-05 & 7.5243e-17 & 1000000 \\
    \bottomrule
  \end{tabular}
  \end{center}
\caption{Comparison between {\arbcd} and \ipm} \label{tab: vsipm}
For {\arbcd}, the average iterations/time/gap/feasibility error are recorded for reference. For {\ipm}, results of four different settings are recorded. The initial reduced problem size is reflected by the subproblem size column.
\end{table}

\paragraph{Comments on Figure~\ref{fig: traj2}} We can observe the following from Figure~\ref{fig: traj2}: although {\sinkhorn} with larger $\epsilon$ may converge fast, the solution accuracy is also lower. In fact, this is true for all Sinkhorn-based algorithms because the optimization problem is not exact - it has an extra entropy term. Therefore, the larger $\gamma$ or $\epsilon$ is chosen, the less accurate the solution becomes. The discrepancy in accuracy does matter, as can be seen by inspecting the solution quality in Figure~\ref{fig: optplan}. On the other hand, when $\epsilon$ is set smaller, the convergence of {\sinkhorn} becomes slower. As can be seen from the plots, when $\epsilon = 0.1$ or $0.01$, {\sinkhorn} converges faster than {\arbcd}; when $\epsilon = 10^{-3}$, {\sinkhorn} is comparable to {\arbcd}; when $\epsilon = 10^{-4}$, {\sinkhorn} is slower than {\arbcd}. In conclusion, if relatively higher precision is desired, {\arbcd} is comparable with {\sinkhorn}. Moreover, note that here we solve the subproblems in {\arbcd} using {\texttt Matlab} built-in solver {\texttt linprog}. {\arbcd} can be faster if more efficient subproblem solvers are applied.

\subsubsection{Comparison with an interior point-inspired algorithm}\label{subsec: ipm}
In this experiment we continue using the 8 datasets, except that Dataset 1 has $n = 1000$.
\paragraph{Methods} \rev{Implementation of {\arbcd} and {\ipm} are specified as follows.\\
{\arbcd.} Algorithm~\ref{alg:arbcd}. Let $m = 150$, $p = \lfloor m^2/n \rfloor$, $s = 0.1$ and $T = 10$. We stop the algorithm when $(f_k - f^*)/f^* \le 10^{-3}$ or when iteration reaches the maximum 10000. Other settings are similar to the experiment in Section~\ref{subsec: sink}. The algorithm outputs the last iterate.\\
{\bf IPM}. A recent interior point-inspired algorithm proposed in \cite{zanetti2023interior}. The Matlab code is directly from the github page in this paper (https://github.com/INFORMSJoC/2022.0184). We have made minimal modification for comparison. We stop the algorithm when $(f_k - f^*)/f^* \le 10^{-3}$ or when the iteration hits 2000. Run four different settings separately, where the initial reduced problem's sizes are approximately $11(2n-1)$, $n^2/10$, $n^2/2$ and $n^2$ ($n = 1000$) correspondingly. For fairness, we also project the solution in each iteration onto the feasible region via the rounding procedure. The algorithm outputs the solution with the best function value gap.}

\paragraph{Comments on Table~\ref{tab: vsipm}} \rev{We could see that the performance of {\arbcd} is better than \ipm$_1$, comparable to {\ipm}$_2$ and {\ipm}$_3$. \ipm$_4$ has the best performance but its subproblem size is also large. In fact, our observation of the memory storage is \ipm$_4$ $\ge$ \ipm$_3$ $\ge$ \ipm$_2$ $\ge$ \ipm$_1$ $\approx$ {\arbcd}. This is consistent with the subproblem size column. In practice, memory consumption of \ipm$_4$ is large, contradicting the goal of the authors in \cite{zanetti2023interior} to save memory. In \cite{zanetti2023interior}, the authors suggest the setting of \ipm$_1$. We also find that when implementing \ipm$_1$,\ipm$_2$,\ipm$_3$, the algorithms develop numerical instabilities, which is reported as N/A in the table. 
This is consistent with the result in \cite{zanetti2023interior} (Figure~4) where ratio of the IPM algorithm finding an ideal solution is below $1$.}

\begin{figure}[ht!]
    \begin{center}
    \includegraphics[width = 1\textwidth]{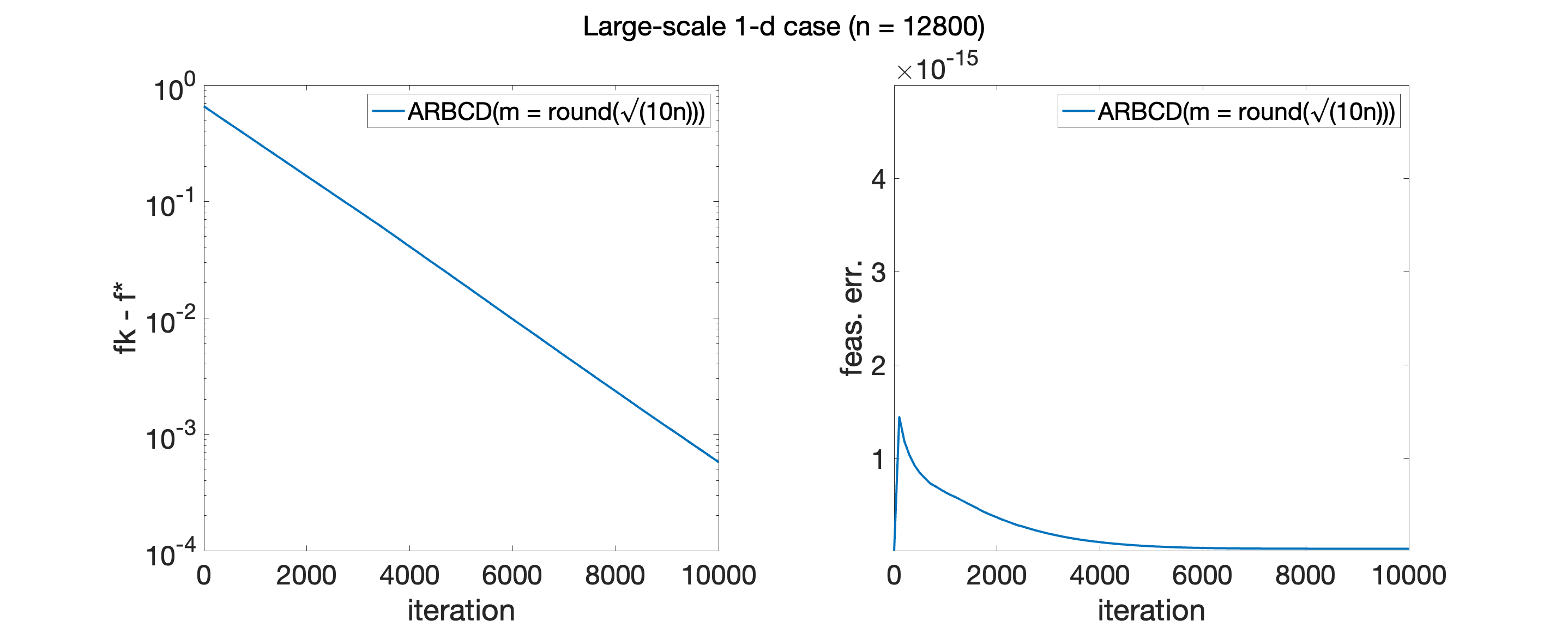}
    \end{center}
    \footnotesize
    \caption{Solving a large-scale problem via {\arbcd}}
    We apply {\arbcd} to solve the large-scale 1-d problem ($n = 12800$). \rev{y-axis of the left plot shows the optimality gap $f_k - f^*$, y-axis of the right plot shows the feasibility error $\| X^k \cdot {\bf 1} - r^1 \| + \| (X^k)^T \cdot {\bf 1} - r^2 \|$, and x-axis records the iteration number.} {\arbcd} is repeated for 3 times and average results are reported. 
    \label{fig: largescale}
\end{figure}

\subsection{Test on a large-scale OT problem}\label{subsec: large}
In this subsection, we generate a pair of 1-dim probability distributions with large discrete support sets ($n = 12800$). For the first distribution, locations of the discrete support ($x^i,i=1,\hdots,n$) are evenly aligned between $[-1,1]$, and their weights/probability are uniformly distributed (i.e., $1/n$). For the other distribution, locations of the discrete support are determined as $\tilde x^i = x^{\sigma(i)} + u^i$, where $\sigma(i)$ is a random permutation of $ i = 1,\hdots, n$, and $u^i$ is a random variable that conforms to a uniform distribution over $[-0.5,0.5]$. Weights/probability are determined as $w_i = \frac{\phi(\tilde x^i)}{\sum_{i=1}^n \phi(\tilde x^i)}$, where $\phi(x)$ is the pdf of the standard normal. The benchmark optimal solution is quickly computed via a closed-form formula for 1-d OT problem ($f^* = 6.236 \times 10^{-3}$). For {\arbcd}, we use the setting $m = \lceil \sqrt{10 n} \rceil$, $p = \lfloor m^2/n \rfloor$, $s = 0.1$ and $T = 10$.
\paragraph{Comments on Figure~\ref{fig: largescale}} The figure showcases the average behavior of {\arbcd} within $10000$ iterations. It is able to locate a solution such that $(f_k - f^*)/f^* \le 0.1$. The convergence is linear by observing the trajectory. We also want to point out that Gurobi 10.01 (academic license) runs out of memory on the desktop we use for numerical experiments. Indeed, saving memory is one of the merits that motivate us to consider RBCD methods.

\section{Conclusion}\label{sec:conclusion}

In this paper, we investigate the RBCD method to solve LP problems, including OT problems. In particular, an expected Gauss-Southwell-$q$ rule is proposed to select the working set $\sI_k$ at iteration $k$. It guarantees almost sure convergence and linear convergence rate and is satisfied by all algorithms proposed in this work. We first develop a vanilla RBCD, called {\rbcd}, to solve general LP problems. Its inexact version is also investigated. Then, by examining the structure of the matrix $A$ in the linear system of OT, characterizing elementary matrices of $\nul(A)$ and identifying conformal realization of any matrix $D \in \nul(A)$, we refine the working set selection. We use two approaches - diagonal band and submatrix - for constructing $\sI_k$ and employ an acceleration technique inspired by the momentum concept to improve the performance of {\rbcd}. In our numerical experiments, we compare all proposed RBCD methods and verify the acceleration effects as well as the sparsity of solutions. We also demonstrate the gap between the theoretical convergence rate and the practical one. We discuss the choice of the subproblem size. Furthermore, we run {\arbcd}, the best among all other methods, against others and also on large-scale OT problems. The results show the advantages of our method in finding relatively accurate solutions to OT problems and saving memory. For future work, we plan to extend our method to handle continuous measures and further improve it through parallelization and multiscale strategies, among other approaches.

\section*{Acknowledgements}
\noindent
The research of YX is supported by Guangdong Province Fundamental and Applied Fundamental Research Regional Joint Fund, (Project 2022B1515130009), and the start-up fund of HKU-IDS. The research of ZZ is supported by Hong Kong RGC grants (Projects 17300318 and 17307921), the National Natural Science Foundation of China  (Project 12171406), Seed Fund for Basic Research (HKU), an R\&D Funding Scheme from the HKU-SCF FinTech Academy, and Seed Funding for Strategic Interdisciplinary Research Scheme 2021/22 (HKU). 

\section*{Declarations}
The authors have no competing interests to declare that are relevant to the content of this article.

\section*{Data availability}
The datasets generated during the current study are available in the GitHub repository, \\
\texttt{https://github.com/yue-xie/RBCDforOT}.

\appendix

\section{Proof of Lemma~\ref{lm: prange}}

\begin{proof} 
Suppose that
\begin{align}\label{ineq1: pr}
    \log((n^2)!/(n^2 - np)! ) \ge 2 \log(n!) + \log(np)!
\end{align}
Then
\begin{align*}
\begin{array}{rrl}
& (n^2)!/(n^2 - np)! )/(np)! & \ge (n!)^2 \\
    \implies & \binom{n^2}{np} / (n!)^2 & \ge 1 \\
    \implies & \frac{\binom{n^2}{np}}{(n!)^2} \cdot \frac{ n(p-2)(n^2- np +1) }{ n^2 - 3 } & \ge 1 \\
    \implies & \frac{n(p-2)}{(n^2 - 3)(n!)^2} & \ge \frac{1}{\binom{n^2}{np}(n^2 - np +1)},
\end{array}
\end{align*}
where the third inequality holds because $p \le n/2$ and $n \ge p \bar K \ge 6$. So we only need to prove \eqref{ineq1: pr}. Note that
\begin{align}
\notag
    \log \frac{(n^2)!}{(n^2 - np)!}  & = \sum_{x = n^2 - np + 1}^{n^2} \log(x) \ge \int_{n^2 - np}^{n^2} (\log x ) dx \\
    \notag
    & = n^2 \log(n^2) - n^2 - \left( (n^2-np)\log(n^2 - np) - n^2 + np\right) \\
    \notag
    & = n^2 \log(n^2) - (n^2 - np) \log(n^2 - np) - np \\
     \notag
    & \overset{ (p = n/K) }{ = } 2np\log n + \frac{K-1}{K} \cdot n^2 \cdot \log\frac{K}{K-1} - np \\
    \label{ineq2: pr}
    & \ge 2np\log n + \frac{2K-3}{2K-2} \cdot np - np.
\end{align}
The last inequality holds because $\log (1+x) \ge x - x^2/2$ for $ x \in (0,1)$ and $p = n / K$. Meanwhile, right hand side of \eqref{ineq1: pr} satisfies the following:
\begin{align}
\notag
    & \quad 2 \log(n!) + \log(np)! \\
    \notag
    & \le 2(n+1) \log(n+1) - 2n + (np+1) \log (np + 1) - np \\
    \notag
    & \le 2(n+1) (\log n + \log 2) - 2n + (np+1) (\log (np) + \log 2) - np \\
    \notag
    & = (np + 2n + 3) \log n + (np+1) \log p + 2(n+1) \log 2 - 2n - np + (\log 2)(np+1) \\
    \notag
    & \overset{\left( p = \frac{n}{K} \right)}{=} 2np \log n + (2n+4) \log n + (\log 4) n + (\log 2)np + \log 8 - 2n - (1+ \log K)np - \log K \\
    \label{ineq3: pr}
    & \overset{ (K \ge \bar K \ge 2, n \ge 6) }{\le} 2np \log n + (2n+4) \log n + (\log 2)np - (1+ \log K)np 
\end{align}
In order to show \eqref{ineq1: pr}, we only need to confirm $\eqref{ineq3: pr} \le \eqref{ineq2: pr}$. By observation, this is equivalent to
\begin{align*}
\begin{array}{crl}
    & \left( \frac{2K-3}{2K-2} + \log \left( \frac{K}{2} \right) \right) np & \ge (2n + 4) \log n \\
    \overset{(p \ge \eta \log n, \bar K \le K)}{\Longleftarrow} & \left( \frac{2 \bar K-3}{2 \bar K-2} + \log \left( \frac{\bar K}{2} \right) \right) \eta n & \ge 2n + 4 \\
    \Longleftrightarrow &  \frac{4}{\left( \frac{2 \bar K-3}{2 \bar K-2} + \log \left( \frac{\bar K}{2} \right) \right) \eta - 2} & \le n.
    \end{array}
\end{align*}
The last inequality is assumed.\qed
\end{proof}

\section{A counterexample of interest}\label{app: counterex}
\begin{exmp}
Consider LP problem \eqref{LP-OT} with $n = 3$, $r^1 = r^2 = (1/3,1/3,1/3)^T$. Let $$C = \pmat{ (1+\epsilon_1)^2 & (2- \epsilon_3)^2 & (1-\epsilon_2)^2 \\ (1- \epsilon_2)^2 & (1 + \epsilon_1)^2 & (2 - \epsilon_3)^2 \\ (2 - \epsilon_3)^2 & (1 - \epsilon_2)^2 & (1 + \epsilon_1)^2 },$$ where $0 < \epsilon_i \ll 1$, $i = 1,2,3$ such that $2 ( 1 + \epsilon_1 )^2 < (1 - \epsilon_2)^2 + (2 - \epsilon_3)^2$. It can be easily seen that the optimal solution is $\gamma^* = \frac{1}{3}\pmat{ 0&0&1\\1&0&0\\0&1&0 }$. Suppose that $\gamma^0 = \frac{1}{3}\pmat{1&0&0\\0&1&0\\0&0&1}$. If we use the submatrix approach \eqref{submat} with $m = 2$ (largest number less than $n$) to select a working set $\sI_k$, then the algorithm will be stuck at $\gamma^0$. It is not globally convergent. This cost matrix corresponds to the following case of transporting a three-point distribution to another one:
\begin{figure}[H]
\begin{center}
\begin{tikzpicture}[scale=1.7]

\draw[dashed, thick] (0,0) -- (1,0);
\draw[dashed, thick] (1,0) -- (1.5,1.732/2);
\draw[dashed, thick] (1.5,1.732/2) -- (1,1.732);
\draw[dashed, thick] (1,1.732) -- (0,1.732);
\draw[dashed, thick] (0,1.732) -- (-0.5,1.732/2);
\draw[dashed, thick] (-0.5,1.732/2) -- (0,0);


\node [left] at (-0.5,1.732/2) {$\tilde{y}_1$};
\node [above] at (0.2,1.732) {$y_1$};
\node [above] at (1,1.732) {$\tilde{y}_2$};
\node [right] at (1.4,1.732/2.5) {$y_2$};
\node [right] at (1,0) {$\tilde{y}_3$};
\node [left] at (-0.1,0.1732) {$y_3$};
\draw [fill=green ,black](-0.1,0.1732) circle(0.6mm);

\draw [fill=green ,black](1,0) circle(0.6mm);

\draw [fill=green ,black](-0.5,1.732/2) circle(0.6mm);

\draw [fill=green ,black](0.2,1.732) circle(0.6mm);
\draw [fill=green ,black](1,1.732) circle(0.6mm);
\draw [fill=green ,black](1.4,1.732/2.5) circle(0.6mm);
\end{tikzpicture}
\end{center}
\caption{A counterexample of interest}
The dashed hexagon has an edge length of $1$. Transport point distribution $\tilde y_1$, $\tilde y_2$ and $\tilde y_3$ to that of $y_1$, $y_2$ and $y_3$.
\end{figure}
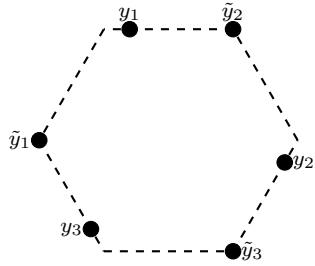
\end{exmp}


%
%

\bibliographystyle{spmpsci}      
\bibliography{ref}   

%
%

\end{document}